\documentclass[12pt,reqno]{amsart}

\usepackage{mathrsfs}
\usepackage[textwidth=16cm,hmarginratio=1:1]{geometry}
\usepackage[colorlinks=true,linkcolor=blue,citecolor=blue,urlcolor=blue]{hyperref}
\theoremstyle{plain}
\newtheorem{theorem}{Theorem}[section]
\newtheorem{proposition}[theorem]{Proposition}
\newtheorem{lemma}[theorem]{Lemma}
\newtheorem{definition}[theorem]{Definition}
\newtheorem{corollary}[theorem]{Corollary}
\theoremstyle{remark}
\newtheorem{remark}[theorem]{Remark}
\newtheorem{example}[theorem]{Example}
\numberwithin{equation}{section}

\newcommand{\K}{\mathbb K}
\newcommand{\A}{\mathbb A}
\newcommand{\Ahat}{\widehat A}
\newcommand{\D}{D}
\newcommand{\wt}{\operatorname{wt}}

\newcommand{\T}{\mathbf T}

\title[Cluster structures on bosonic extensions]{Quantum cluster structures
on bosonic extensions}
\author{Yingjin Bi}
\address{Department of Mathematics, Harbin Engineering University;
Research Institute for Mathematical Sciences, Kyoto University}
\email{yingjinbi@mail.bnu.edu.cn}
\date{}

\begin{document}
\raggedbottom

\begin{abstract}
We construct a quantum cluster algebra structure on the standard
$\mathbb Z[q^{\pm1/2}]$-form of the bosonic extension algebra
$\widehat A(b)$ for every positive braid $b$ and every symmetrizable
generalized Cartan matrix. The proof realizes this algebra as the
degree-zero part of a localization of a larger quantum unipotent
coordinate ring, obtained by thickening a word for $b$.
For symmetric Cartan matrices, this realization also gives a monoidal
categorification of the quantum cluster algebra.
\end{abstract}
\maketitle

\setcounter{tocdepth}{1} 
\tableofcontents
\section{Introduction}

Bosonic extensions of quantum unipotent coordinate rings arose
in the work of Hernandez--Leclerc on quantum Grothendieck rings
and derived Hall algebras in type $ADE$
\cite{hernandez2015quantum}. Kashiwara--Kim--Oh--Park developed
their braid symmetries and global bases for arbitrary
symmetrizable generalized Cartan matrices
\cite{KKOP-braid,KKOP-global}; see also \cite{oh2025pbw} for
the PBW theory in finite type. The braid symmetries associate a subalgebra
$\widehat A(b)$ with each positive braid $b$.
We study the quantum cluster structure on its standard
$\mathbb Z[q^{\pm1/2}]$-form and the monoidal categorification
of this structure. For a word $\mathbf i$ representing $b$,
let $\mathcal A_q(\mathbf s(\mathbf i))$ denote the quantum
cluster algebra of the word seed $\mathbf s(\mathbf i)$.
The frozen variables of this word seed are not inverted.

Qin constructed common triangular bases for these quantum
cluster algebras and obtained quasi-categorifications in
symmetric type \cite{qin2024analogs}, allowing changes of
quantization and coefficients, as well as localization.
In finite Cartan type, his braid-group construction identifies
$\mathcal A_q(\mathbf s(\mathbf i))$ with $\widehat A(b)$
\cite{qin2024based}. In type $ADE$, Kashiwara--Kim--Oh--Park
gave monoidal categorifications in Hernandez--Leclerc
categories \cite{kashiwara2025monoidal}. The author also
proved the cluster structure in this case using transition
maps between Lusztig parametrizations \cite{bi2025cluster}.

We extend the integral cluster structure known in type $ADE$
\cite{kashiwara2025monoidal} to all symmetrizable generalized Cartan
matrices, including affine and indefinite types. In symmetric type,
we also construct a monoidal categorification. Our approach draws on the
thickening construction of Bao--He \cite{bao2022total} and
Qin's extension of words \cite[Lemma~6.10]{qin2024analogs}.
Qin inserts auxiliary letters to obtain a reduced word for
a larger Weyl group and uses this construction to establish
quasi-categorification results.

We adapt the thickening construction of Bao--He~\cite{bao2022total}
to the chosen symmetrizing integers and identify the bosonic extension algebra $\Ahat(b)$ with a
subalgebra of a localized quantum unipotent coordinate ring
for a larger symmetrizable Kac--Moody algebra. The integral quantum
cluster structure of this ring, established by Goodearl--Yakimov
\cite{GY-integral}, yields the required structure on $\Ahat(b)$.
In symmetric type, we also use the ambient monoidal categorification
to categorify the resulting quantum cluster algebra.

The first difficulty is to recover $\widehat A(b)$, together
with its standard integral form, from the larger quantum
unipotent coordinate ring. We invert the auxiliary frozen
variables and take the subalgebra of degree zero in the
auxiliary roots. Using CGL extensions, we prove that this
subalgebra is isomorphic to $\widehat A(b)$ and that the
isomorphism preserves the standard integral forms.

The second difficulty is to obtain the required cluster
structure inside this subalgebra, since the ambient cluster
variables need not have degree zero. We normalize them
using the auxiliary frozen variables and prove that this
normalization respects every mutation. The normalized
variables generate the whole integral subalgebra and
identify it with $\mathcal A_{\A}(\mathbf s(\mathbf i))$.\\

\noindent
We now describe the construction. Let $q^{1/2}$ be an
indeterminate, and set $\A=\mathbb Z[q^{\pm1/2}]$ and
$\mathbb K=\mathbb Q(q^{1/2})$.
Let $C=(c_{ij})_{i,j\in I}$ be a symmetrizable generalized Cartan
matrix with finite index set $I$. Fix positive integers $d_i$ such that
$d_ic_{ij}=d_jc_{ji}$, and set $q_i=q^{d_i}$.
We normalize the invariant form by
$(\alpha_i,\alpha_j)=d_ic_{ij}$ and
$(\alpha_i,\varpi_j)=d_i\delta_{ij}$.
Let $\mathbf i=(i_1,\ldots,i_r)$ be a nonempty positive word
representing $b=\sigma_{i_1}\cdots\sigma_{i_r}$.
Adjoin vertices $\infty_1,\ldots,\infty_{r-1}$ and extend $C$
to a symmetrizable generalized Cartan matrix $\widetilde C$ by
\[
\widetilde c_{i,\infty_p}=-2,\qquad
\widetilde c_{\infty_p,i}=-2d_i,\qquad
\widetilde c_{\infty_p,\infty_q}=-2\quad(p\ne q),
\]
with diagonal entries $2$. Its symmetrizing integers are
$\widetilde d_i=d_i$ and $\widetilde d_{\infty_p}=1$.
Then
\[
\operatorname{th}(\mathbf i)
=(i_1,\infty_1,i_2,\ldots,\infty_{r-1},i_r)
\]
is a reduced word for an element $\widetilde w_{\mathbf i}$
of the enlarged Weyl group. Let
$A_{\mathbf i,\A}=A_{\A}(\mathfrak n(\widetilde w_{\mathbf i}))$
be its dual canonical integral quantum coordinate ring, and
let $X_1,\ldots,X_{2r-1}$ be the dual PBW generators.
For $1\le p<r$, set $H_p=q^{-1/2}X_{2p}$; this is the
normalized frozen seed variable at $\infty_p$.

Define the separator degree on the enlarged root lattice by
$\deg_\infty\alpha_i=0$ for $i\in I$ and
$\deg_\infty\alpha_{\infty_p}=e_p$, where $e_p$ is the $p$th
standard basis vector of $\mathbb Z^{r-1}$.
This grading extends to the Ore localization at the $H_p$.
Set
\[
L_{\mathbf i,\A}
=A_{\mathbf i,\A}[H_1^{-1},\ldots,H_{r-1}^{-1}],
\qquad
\D_{\mathbf i,\A}=(L_{\mathbf i,\A})_0,
\]
where the subscript $0$ denotes the component of separator
degree zero. The degrees of the $H_p$ form a
$\mathbb Z$-basis of $\mathbb Z^{r-1}$, so each odd PBW
generator has a unique right normalization
$Y_k=X_{2k-1}H^{\mathbf m_k}$ of separator degree zero.
The ordered monomials in the $Y_k$ form an $\A$-basis of
$\D_{\mathbf i,\A}$. For the empty word, all these integral
algebras are $\A$.

\begin{theorem}[{Theorem~\ref{thm:integral-comparison}}]
\label{thm:intro-realization}
For every symmetrizable generalized Cartan matrix and every
positive word $\mathbf i$ representing $b$, there are units
$c_k\in\A^\times$ such that
\[
\widehat A_{\A}(b)\xrightarrow{\sim}\D_{\mathbf i,\A},
\qquad P_k^{\mathbf i}\longmapsto c_kY_k.
\]
Here $\widehat A_{\A}(b)$ is the standard bosonic integral
form spanned by its global basis, or equivalently by its
normalized PBW basis.
\end{theorem}

To obtain the cluster structure, we normalize the ambient
quantum cluster variables by Laurent monomials in the
separators. This normalization commutes with quantum mutation
and gives the word seed $\mathbf s(\mathbf i)$ of
Section~\ref{app:word-seeds}. The resulting integral cluster
algebra is $\D_{\mathbf i,\A}$, so the preceding theorem
gives the following isomorphism.

\begin{theorem}[{Corollary~\ref{cor:isoAb}}]
\label{thm:intro-cluster}
For every symmetrizable generalized Cartan matrix and every
positive word $\mathbf i$ representing $b$, there is an
isomorphism of $\A$-algebras
\[
\widehat A_{\A}(b)
\cong\mathcal A_{\A}(\mathbf s(\mathbf i)).
\]
\end{theorem}

For the categorification statements below, assume that $C$ is symmetric
and choose $d_i=1$ for every $i$. Then $\widetilde C$ is symmetric. Let
$\mathscr C_{\widetilde w_{\mathbf i}}$ be the standard
category of finite-dimensional graded modules over a symmetric
quiver Hecke algebra for $\widetilde C$, over an algebraically
closed field of characteristic zero. Localize this category
by inverting the determinantial modules corresponding to
the separators $H_p$, and let $\mathscr C_{\mathbf i}^{0}$
be the full subcategory of separator degree zero. Write
$K_{0,\A}(\mathscr C)=
\A\otimes_{\mathbb Z[q^{\pm1}]}K_0(\mathscr C)$.
A simple object is called real if its convolution square
is simple.

\begin{theorem}[{Theorem~\ref{thm:categorification}}]
\label{thm:intro-categorification}
For every symmetric generalized Cartan matrix and every positive word
$\mathbf i$, there is an isomorphism
\[
\mathcal A_{\A}(\mathbf s(\mathbf i))
\cong K_{0,\A}(\mathscr C_{\mathbf i}^{0})
\]
under which every quantum cluster monomial is a power of
$q^{1/2}$ times the class of a real simple object in
$\mathscr C^0_{\mathbf i}$. Thus $\mathscr C_{\mathbf i}^{0}$
is a monoidal categorification of
$\mathcal A_{\A}(\mathbf s(\mathbf i))$.
\end{theorem}

Continue to assume that $C$ is symmetric. Let $\delta(\mathbf i)$
be the Demazure product of
$\mathbf i$, and let $v\leq\delta(\mathbf i)$.
Denote by $\mathbf s(v,\mathbf i)$ the
Bao--Ye--M\'enard seed defined in
Section~\ref{app:menard-seed}.
All categories below are associated with the enlarged
Cartan matrix $\widetilde C$.
Let $M_p$ be the determinantial module whose class
corresponds to $H_p$ up to a power of $q^{1/2}$.
Note that $M_p\in\mathscr C_{\widetilde w_{\mathbf i},v}$. Here $\mathscr C_{\widetilde w_{\mathbf i},v}$ refers to the category defined in Section~\ref{app:monoidal-localization}.
Define
\[
\mathscr L_{\widetilde w_{\mathbf i},v}
:=\mathscr C_{\widetilde w_{\mathbf i},v}
[M_p^{-1}\mid 1\leq p<r],
\qquad
\mathscr C_{\mathbf i,v}^{0}
:=(\mathscr L_{\widetilde w_{\mathbf i},v})_0,
\]
where the subscript $0$ denotes the subcategory of
auxiliary degree zero. 

\begin{theorem}[{Theorem~\ref{thm:menard-seed-embedding}}]
Under the symmetric-type assumptions above, normalization induces an algebra embedding
\[
\Phi_{\mathbf i,v,\mathbb A}:
\mathcal A_{\mathbb A}(\mathbf s(v,\mathbf i))
\hookrightarrow
K_{0,\mathbb A}(\mathscr C_{\mathbf i,v}^{0})
.
\]
It sends every quantum cluster monomial to a power
of $q^{1/2}$ times the class of a simple object of
$\mathscr C_{\mathbf i,v}^{0}$.
\end{theorem}

Section~\ref{sec:CGL} recalls CGL extensions and quantum
unipotent coordinate rings. Section~\ref{sec:bosonic} studies
the bosonic integral form and the straightening corrections.
Section~\ref{sec:new-subalgebra} constructs the thickening
and proves the integral comparison. In
Section~\ref{sec:cluster-category}, we prove the cluster
isomorphism in symmetrizable type and the monoidal categorification
in symmetric type, and give an explicit example. Appendix~\ref{sec:appendix} collects the
conventions for quantum seeds, word seeds, and monoidal
localization.

\medskip
\noindent\textbf{Use of AI.}
The author used ChatGPT 6 Astra to edit the manuscript and
assist in checking the proofs. The author takes full
responsibility for the arguments and results.

\subsection*{Acknowledgments}

The author thanks Masaki Kashiwara for helpful discussions
on the results of this paper and Ryo Fujita for suggestions
that helped improve them. This work was supported by the
China Scholarship Council (CSC), No.~202506680034.

\section{CGL extensions}\label{sec:CGL}

Throughout this section, \(\mathbb K\) is a field and all algebras
are unital \(\mathbb K\)-algebras.
For integers \(a\le b\), write \([a,b]:=\{a,a+1,\ldots,b\}\).
Let \(S\) be an algebra and \(\sigma\) a \(\mathbb K\)-algebra
endomorphism of \(S\).
A \emph{left \(\sigma\)-derivation} is a \(\mathbb K\)-linear map
\(\delta:S\to S\) satisfying
\(\delta(st)=\sigma(s)\delta(t)+\delta(s)t\) for all \(s,t\in S\).
The \emph{Ore extension} \(S[x;\sigma,\delta]\) is the free left
\(S\)-module with basis \(\{x^n\mid n\ge0\}\), whose multiplication
extends that of \(S\) and is determined by
\(xs=\sigma(s)x+\delta(s)\).

In our applications, \(\mathbb K=\mathbb Q(q^{1/2})\).
We treat integral forms over \(\A=\mathbb Z[q^{\pm1/2}]\)
separately. The prime classification below applies over
\(\mathbb K\); over \(\A\), scalar primes give additional
homogeneous prime elements.

\subsection{Definition of CGL extensions}

Consider the iterated Ore extension
\begin{equation}\label{eq:R}
  R=\mathbb K[x_1][x_2;\sigma_2,\delta_2]\cdots
    [x_N;\sigma_N,\delta_N].
\end{equation}
Let \(R_k\) be the subalgebra generated by \(x_1,\ldots,x_k\),
and set \(R_0:=\mathbb K\), \(\sigma_1:=\operatorname{id}_{\mathbb K}\),
and \(\delta_1:=0\).

\begin{definition}
The algebra \(R\) in \eqref{eq:R} is a
\emph{Cauchon--Goodearl--Letzter (CGL) extension} if it admits a grading
\(R=\bigoplus_{\gamma\in Q}R_\gamma\) by a finite-rank lattice \(Q\)
with the following properties. Let \(\mathcal H\) be the split torus
with character lattice \(Q\). Identify its \(\mathbb K\)-points
with \(\operatorname{Hom}_{\mathbb Z}(Q,\mathbb K^\times)\), acting
on \(R\) by \(h\cdot u=h(\gamma)u\) for \(u\in R_\gamma\).
\begin{enumerate}
  \item Each generator \(x_k\) is homogeneous.
  \item For \(k\in[2,N]\), the derivation \(\delta_k\) is locally
  nilpotent on \(R_{k-1}\).
  \item For each \(k\in[1,N]\), there exists \(h_k\in\mathcal H\)
  such that \(\sigma_k=(h_k\cdot)|_{R_{k-1}}\) and
  \(h_k\cdot x_k=\lambda_kx_k\), where
  \(\lambda_k\in\mathbb K^\times\) is not a root of unity.
\end{enumerate}
\end{definition}

In particular, \(\sigma_k(x_j)=\lambda_{kj}x_j\) for some
\(\lambda_{kj}\in\mathbb K^\times\) whenever \(1\le j<k\le N\).

Here \(\lambda_k\) denotes the eigenvalue of \(h_k\) on \(x_k\).
The multiplicatively skew-symmetric matrix \((\lambda_{kj})\)
has diagonal entries \(1\).

\subsection{Homogeneous prime elements}

A \emph{prime element} of a domain \(A\) is a nonzero nonunit
\(u\in A\) such that \(Au=uA\) and \(A/Au\) is a domain.
A prime element is \emph{homogeneous} if it is homogeneous
for the given grading.

For a function \(\eta:[1,N]\to\mathbb Z\), define its predecessor
and successor functions by
\(p(k):=\max\{j<k\mid\eta(j)=\eta(k)\}\) and
\(s(k):=\min\{j>k\mid\eta(j)=\eta(k)\}\), with
\(\max\emptyset:=-\infty\) and \(\min\emptyset:=+\infty\).
For \(m\in[1,N]\), set
\(\mathcal P(m):=\{j\in[1,m]\mid s(j)>m\}\).

\begin{theorem}[{\cite[Theorem~4.3]{GY-UFD};
\cite[Theorem~3.6]{GY-cluster}}]
\label{thm:cgl-recursive-primes}
Let \(R\) be a CGL extension.
There exist a function \(\eta:[1,N]\to\mathbb Z\) and nonzero
elements \(c_k\in R_{k-1}\), for \(p(k)\ne-\infty\), such that
the recursively defined elements
$$
  y_k:=
  \begin{cases}
    x_k, & p(k)=-\infty,\\
    y_{p(k)}x_k-c_k, & p(k)\ne-\infty
  \end{cases}
$$
are homogeneous and, for every \(m\in[1,N]\),
\(\{y_j\mid j\in\mathcal P(m)\}\) is a complete set of representatives
of the homogeneous prime elements of \(R_m\), up to nonzero scalar
multiples.
Moreover, \(p(k)=-\infty\) if and only if \(\delta_k=0\).

The elements \(y_1,\ldots,y_N\) and the partition into the level
sets of \(\eta\) are uniquely determined by the CGL presentation.
\end{theorem}

We call the level sets of \(\eta\) the \emph{prime chains}.
Their terminal indices in \([1,m]\) form \(\mathcal P(m)\).

\begin{theorem}[{\cite[Theorem~4.2 and Proposition~3.8]{GY-UFD}}]
\label{thm:cgl-detected-prime}
With the notation of Theorem~\ref{thm:cgl-recursive-primes},
suppose that \(k\in[2,N]\) and \(\delta_k\ne0\). Then
\(p(k)\in\mathcal P(k-1)\) and
\(\delta_k(y_j)=0\) for every
\(j\in\mathcal P(k-1)\setminus\{p(k)\}\).
Writing \(\sigma_k(y_{p(k)})=a_ky_{p(k)}\), with
\(a_k\in\mathbb K^\times\), we have
$$
  \delta_k(y_{p(k)})=a_k(\lambda_k-1)c_k\ne0.
$$
Hence \(y_{p(k)}\) is, up to a nonzero scalar, the unique homogeneous
prime element of \(R_{k-1}\) on which \(\delta_k\) is nonzero.
The homogeneous prime elements of \(R_k\) are obtained by replacing
\(y_{p(k)}\) with \(y_k=y_{p(k)}x_k-c_k\) and retaining the other
representatives indexed by \(\mathcal P(k-1)\).
\end{theorem}

\subsection{Symmetric CGL extensions}

For \(1\le a\le b\le N\), let \(R[a,b]\) be the subalgebra of \(R\)
generated by \(x_a,\ldots,x_b\), and set \(R[a,b]:=\mathbb K\)
when \(a>b\).
Thus \(R[1,k]=R_k\).

\begin{definition}
A CGL extension \(R\) is \emph{symmetric} if:
\begin{enumerate}
  \item \(\delta_k(x_j)\in R[j+1,k-1]\) for all
  \(1\le j<k\le N\);
  \item for each \(j\in[1,N]\), there exists \(h_j^*\in\mathcal H\)
  such that \(h_j^*\cdot x_k=\lambda_{kj}^{-1}x_k\) for all \(k>j\)
  and \(h_j^*\cdot x_j=\lambda_j^*x_j\), where
  \(\lambda_j^*\in\mathbb K^\times\) is not a root of unity.
\end{enumerate}
\end{definition}

\begin{lemma}[{\cite[Corollary~5.7]{GY-cluster}}]
\label{lem:cgl-interval-chains}
Let \(R\) be a symmetric CGL extension, with \(\eta,p,s\) as in
Theorem~\ref{thm:cgl-recursive-primes}, and fix \(1\le a\le b\le N\).
Then \(R[a,b]\), with its induced Ore presentation, is a symmetric
CGL extension whose chain-label function can be chosen as
\(\eta|_{[a,b]}\).

With the original indexing, its predecessor function satisfies
\(p_{[a,b]}(k)=p(k)\) when \(p(k)\ge a\), and
\(p_{[a,b]}(k)=-\infty\) otherwise.
Similarly, \(s_{[a,b]}(k)=s(k)\) when \(s(k)\le b\), and
\(s_{[a,b]}(k)=+\infty\) otherwise.
Its recursive homogeneous prime elements are
$$
  y_k^{[a]}:=
  \begin{cases}
    x_k, & p(k)<a,\\
    y_{p(k)}^{[a]}x_k-c_k^{[a]}, & p(k)\ge a,
  \end{cases}
  \qquad
  c_k^{[a]}\in R[a,k-1].
$$
These elements may differ from \(y_k\).
For \(a<k\le b\), we have
$$
  \left.\delta_k\right|_{R[a,k-1]}\ne0
  \quad\Longleftrightarrow\quad
  p(k)\ge a.
$$
\end{lemma}

The last equivalence follows by applying
Theorem~\ref{thm:cgl-recursive-primes} to the induced presentation
of \(R[a,k]\).

Let
\(R=\mathbb K[X_1][X_2;\sigma_2,\delta_2]\cdots
[X_N;\sigma_N,\delta_N]\)
be a symmetric CGL extension.
Assume that \(q\in\mathbb K^\times\) and that
\(\Lambda=(\Lambda_{st})\in M_N(\mathbb Z)\) is skew-symmetric
with \(\sigma_t(X_s)=q^{-\Lambda_{st}}X_s\) for \(s<t\).
Write \(X^{\mathbf a}=X_1^{a_1}\cdots X_N^{a_N}\) for
\(\mathbf a\in\mathbb N^N\), where
\(\mathbb N=\mathbb Z_{\geq0}\).

We use the \emph{right lexicographic order} on \(\mathbb N^N\):
for distinct \(\mathbf c,\mathbf d\), write
\(\mathbf c\prec\mathbf d\) if \(c_j<d_j\) at the largest
index \(j\) for which \(c_j\ne d_j\).
Thus we compare coordinates in the order \(N,N-1,\ldots,1\).
Set
\(R_{\prec\mathbf d}
=\operatorname{span}_{\mathbb K}
\{X^{\mathbf c}\mid\mathbf c\prec\mathbf d\}\).
\begin{lemma}[Leading commutation factors of PBW monomials]
\label{lem:leading-commutation}
For all \(\mathbf a,\mathbf b\in\mathbb N^N\),
\begin{equation}\label{eq:PBW-leading-product}
X^{\mathbf a}X^{\mathbf b}
\equiv
q^{-\sum_{s<t}a_t b_s\Lambda_{st}}
X^{\mathbf a+\mathbf b}
\pmod{R_{\prec\mathbf a+\mathbf b}}.
\end{equation}
Consequently,
\begin{equation}\label{eq:PBW-leading-commutation}
X^{\mathbf a}X^{\mathbf b}
-q^{\mathbf a^{\mathsf t}\Lambda\mathbf b}
X^{\mathbf b}X^{\mathbf a}
\in R_{\prec\mathbf a+\mathbf b}.
\end{equation}
If all \(\delta_t\) vanish, both formulas hold without lower
terms. In this case, after adjoining the inverses of the generators, they
also hold for \(\mathbf a,\mathbf b\in\mathbb Z^N\).
\end{lemma}

\begin{proof}
For \(s<t\), the Ore relation and the symmetry condition give
$$
X_tX_s=q^{-\Lambda_{st}}X_sX_t+\delta_t(X_s),
\qquad
\delta_t(X_s)\in
\mathbb K\langle X_{s+1},\ldots,X_{t-1}\rangle,
$$
where the algebra on the right is \(\mathbb K\) if \(t=s+1\).
For a word \(W=X_{j_1}\cdots X_{j_m}\), define
$$
\nu(W)=\sum_{a=1}^{m}\mathbf e_{j_a},
\qquad
\operatorname{inv}(W)
=\#\{(a,b)\mid a<b,\ j_a>j_b\},
$$
with \(\nu(1)=0\) and \(\operatorname{inv}(1)=0\).
The interval condition allows us to write
\(\delta_t(X_s)=\sum_{\ell}d_\ell W_\ell\),
where each \(W_\ell\) is a word in
\(X_{s+1},\ldots,X_{t-1}\).
Thus
$$
\begin{aligned}
\nu(W_\ell)_u&=0=(\mathbf e_s+\mathbf e_t)_u
&&\text{for }u>t,\\
\nu(W_\ell)_t&=0<1=(\mathbf e_s+\mathbf e_t)_t,
\end{aligned}
\qquad
\Longrightarrow\qquad
\nu(W_\ell)\prec\mathbf e_s+\mathbf e_t.
$$

If \(W=UX_tX_sV\), with \(s<t\), the straightening relation gives
$$
W=q^{-\Lambda_{st}}UX_sX_tV
+\sum_{\ell}d_\ell UW_\ell V.
$$
The leading term satisfies
$$
\nu(UX_sX_tV)=\nu(W),
\qquad
\operatorname{inv}(UX_sX_tV)=\operatorname{inv}(W)-1.
$$
Since \(\prec\) is compatible with addition, each correction
term satisfies
$$
\begin{aligned}
\nu(UW_\ell V)
&=\nu(U)+\nu(W_\ell)+\nu(V)\\
&\prec\nu(U)+\mathbf e_s+\mathbf e_t+\nu(V)
=\nu(W).
\end{aligned}
$$
Hence each replacement strictly decreases the pair
\((\nu(W),\operatorname{inv}(W))\), ordered first by
\(\prec\) and then by the usual order on \(\mathbb N\).
This order is well-founded, so induction gives
$$
W\in\mathbb K X^{\nu(W)}+R_{\prec\nu(W)},
\qquad
UW_\ell V\in R_{\prec\nu(W)}.
$$
In particular, a correction term cannot contribute to the
leading PBW monomial \(X^{\nu(W)}\).

In straightening \(X^{\mathbf a}X^{\mathbf b}\), each of the
\(b_s\) copies of \(X_s\) in the second factor passes the
\(a_t\) copies of \(X_t\) in the first factor whenever \(s<t\).
Each interchange contributes \(q^{-\Lambda_{st}}\).
Thus the leading coefficient is
\(q^{-\sum_{s<t}a_t b_s\Lambda_{st}}\), and the ordered
monomial is \(X^{\mathbf a+\mathbf b}\).
Every contribution involving a correction term lies in
\(R_{\prec\mathbf a+\mathbf b}\), proving
\eqref{eq:PBW-leading-product}.

Interchanging \(\mathbf a\) and \(\mathbf b\), the ratio
of the two leading coefficients is
\(q^{\sum_{s<t}(a_s b_t-a_t b_s)\Lambda_{st}}
=q^{\mathbf a^{\mathsf t}\Lambda\mathbf b}\),
by skew-symmetry of \(\Lambda\).
The leading terms therefore cancel in
\eqref{eq:PBW-leading-commutation}.

If all \(\delta_t\) vanish, no correction terms occur.
After adjoining the inverses of the generators, the relation
\(X_s^uX_t^v=q^{uv\Lambda_{st}}X_t^vX_s^u\)
holds for all \(u,v\in\mathbb Z\), so the same calculation
applies to arbitrary integer exponent vectors.
\end{proof}

\subsection{Cauchon's quantum torus}

Let \(\Lambda=(t_{kl})_{k,l=1}^N\) be the multiplicatively
skew-symmetric matrix defined by \(t_{kk}:=1\),
\(t_{kl}:=\lambda_{kl}\), and \(t_{lk}:=\lambda_{kl}^{-1}\) for \(k>l\).
Thus \(t_{kl}t_{lk}=1\).
For any CGL extension \(R\), Cauchon's deleting-derivations procedure
produces elements \(\overline{x}_1,\ldots,\overline{x}_N\) in the
division ring of fractions \(\operatorname{Frac}(R)\), satisfying
\(\overline{x}_k\overline{x}_l
=t_{kl}\overline{x}_l\overline{x}_k\).
They generate a quantum torus
$$
  \mathcal T_\Lambda
  :=\mathbb K\langle
    \overline{x}_1^{\pm1},\ldots,\overline{x}_N^{\pm1}
  \rangle,
  \qquad
  R\subseteq\mathcal T_\Lambda\subseteq\operatorname{Frac}(R);
$$
see \cite[Section~6.1]{GY-UFD}.

\subsection{Quantum unipotent coordinate rings}

Let $C=(c_{ij})_{i,j\in I}$ be a symmetrizable generalized Cartan
matrix with finite index set $I$, and let $W$ be its Weyl group.
Fix positive integers $d_i$ such that
$D_C C$ is symmetric, where $D_C=\operatorname{diag}(d_i)_{i\in I}$,
and put $q_i=q^{d_i}$.
Put $\mathbb A_0=\mathbb Z[q^{\pm1}]$,
$\A=\mathbb Z[q^{\pm1/2}]$, and $\mathbb K=\mathbb Q(q^{1/2})$.
Write $Q=\bigoplus_{i\in I}\mathbb Z\alpha_i$ and
$Q_+=\bigoplus_{i\in I}\mathbb N\alpha_i$, with
$(\alpha_i,\alpha_j)=d_i c_{ij}$.
We use fundamental weights $\varpi_i$ in a fixed realization
of $C$, so that $(\alpha_i,\varpi_j)=d_i\delta_{ij}$.
For $\beta=\sum_i b_i\alpha_i\in Q$, put
$\operatorname{ht}(\beta)=\sum_i b_i$.

We use the coordinate-ring conventions of
\cite[Sections~1.1--1.3]{KKOP-braid}.
Let $U_q^-(\mathfrak g)$ be the negative half of the quantum
group over $\mathbb Q(q)$, generated by $f_i$, and let
$U^-_{\mathbb A_0}(\mathfrak g)$ be its divided-power integral form.
The weight of $f_i$ is $-\alpha_i$.
The twisted coproduct $r$ is determined by
$r(f_i)=f_i\otimes1+1\otimes f_i$ and the multiplication
\[
 (a\otimes b)(c\otimes d)
 =q^{-(\operatorname{wt}b,\operatorname{wt}c)}ac\otimes bd
\]
on homogeneous tensors.
The quantum coordinate ring is the restricted graded dual
\[
 A_q(\mathfrak n)
 =\bigoplus_{\beta\in Q_+}
 \operatorname{Hom}_{\mathbb Q(q)}
 \bigl(U_q^-(\mathfrak g)_{-\beta},\mathbb Q(q)\bigr),
 \qquad
 \langle uv,x\rangle=\langle u\otimes v,r(x)\rangle.
\]
Its integral form is
\[
 A_{\mathbb A_0}(\mathfrak n)
 =\{u\in A_q(\mathfrak n)\mid
       \langle u,U^-_{\mathbb A_0}(\mathfrak g)\rangle
       \subseteq\mathbb A_0\}.
\]
We extend scalars to $\A$ or $\mathbb K$ when indicated.
For a functional supported on $U_q^-(\mathfrak g)_{-\beta}$,
we set
$\deg u=\beta$ and $\operatorname{wt}u=-\beta$.

Let $(\ ,\ )_L$ be Lusztig's bilinear form on
$U_q^-(\mathfrak g)$, normalized by
$(f_i,f_i)_L=(1-q_i^2)^{-1}$.
It gives the algebra isomorphism
\[
 \iota:U_q^-(\mathfrak g)\longrightarrow A_q(\mathfrak n),
 \qquad \langle\iota(x),y\rangle=(x,y)_L.
\]
In particular, if $u_i$ is the functional dual to $f_i$, then
$\iota(f_i)=(1-q_i^2)^{-1}u_i$.
This agrees with \cite[(1.8) and (1.10)]{KKOP-braid}.

Let $\mathsf S_i=T'_{i,-1}$ be Lusztig's braid automorphism.
For $j\ne i$, we use the convention
\[
 \mathsf S_i(f_j)
 =\sum_{a+b=-c_{ij}}(-q_i)^b f_i^{(a)}f_jf_i^{(b)},
 \qquad f_i^{(a)}=f_i^a/[a]_{q_i}!.
\]
For a reduced word $\mathbf w=(i_1,\ldots,i_r)$ of $w$, put
$w_{\le0}=1$, $w_{\le k}=s_{i_1}\cdots s_{i_k}$,
$\beta_k=w_{\le k-1}\alpha_{i_k}$, and
\[
 F_k^{\mathbf w}
 =\mathsf S_{i_1}\cdots\mathsf S_{i_{k-1}}(f_{i_k}),
 \qquad
 P_k^{\mathbf w}=(1-q_{i_k}^2)\iota(F_k^{\mathbf w}).
\]
Thus $\deg P_k^{\mathbf w}=\beta_k$.
The $\mathbb A_0$-algebra generated by these dual root vectors
is denoted by $A_{\mathbb A_0}(\mathfrak n(w))$; it is independent
of the reduced expression.
We put
$A_{\A}(\mathfrak n(w))=\A\otimes_{\mathbb A_0}
A_{\mathbb A_0}(\mathfrak n(w))$ and use
$A_q(\mathfrak n(w))$ for its scalar extension to $\mathbb K$.

For $\mathbf a=(a_1,\ldots,a_r)\in\mathbb N^r$, define
\[
 P^{\mathbf w}(\mathbf a)
 =q^{\sum_{k=1}^r d_{i_k}\binom{a_k}{2}}
   (P_r^{\mathbf w})^{a_r}\cdots(P_1^{\mathbf w})^{a_1}.
\]
These elements form the dual PBW basis of
$A_{\mathbb A_0}(\mathfrak n(w))$.
The upper global basis is compatible with this subalgebra;
write $G^{\mathbf w}(\mathbf a)$ for its element with
Lusztig parameter $\mathbf a$.
In particular, $G^{\mathbf w}(e_k)=P_k^{\mathbf w}$.
For these facts over both the field and the integral form,
see \cite[Section~1.3 and Section~1.4]{KKOP-braid}.

For $a\le b$, write $A[a,b]$ for the $\mathbb K$-subalgebra
generated by $P_a^{\mathbf w},\ldots,P_b^{\mathbf w}$;
put $A[a,b]=\mathbb K$ if $a>b$.
The Levendorskii--Soibelman relation in these conventions is
\begin{equation}\label{eq:cell-LS-fixed}
 P_j^{\mathbf w}P_k^{\mathbf w}
 -q^{-(\beta_j,\beta_k)}P_k^{\mathbf w}P_j^{\mathbf w}
 \in A[j+1,k-1]
 \qquad(j<k).
\end{equation}
Its coefficients in the interval PBW basis belong to
$\mathbb A_0$; see also \cite[Proposition~4.13]{KKOP-braid}
applied to a reduced braid word.
For homogeneous $u\in A[1,k-1]$, define
\[
 \sigma_k(u)=q^{(\beta_k,\deg u)}u,
 \qquad
 \delta_k(u)=P_k^{\mathbf w}u-\sigma_k(u)P_k^{\mathbf w}.
\]

\begin{proposition}\label{pro:CGL-extension-unipotent}
The algebra $A_q(\mathfrak n(w))$ has the symmetric CGL
presentation
\[
 A_q(\mathfrak n(w))
 =\mathbb K[P_1^{\mathbf w}]
 [P_2^{\mathbf w};\sigma_2,\delta_2]\cdots
 [P_r^{\mathbf w};\sigma_r,\delta_r].
\]
For the root torus
$\mathcal H=\operatorname{Hom}_{\mathbb Z}(Q,\mathbb K^\times)$
acting through the degree grading, one can take
\[
 h_k(\beta)=q^{(\beta_k,\beta)},
 \qquad h_k^*(\beta)=q^{-(\beta_k,\beta)}.
\]
Their eigenvalues on $P_k^{\mathbf w}$ are $q_{i_k}^2$ and
$q_{i_k}^{-2}$, respectively.
Moreover, the same Ore presentation is defined over $\A$:
the maps $\sigma_k^{\pm1}$ and $\delta_k$ preserve the
corresponding integral prefix algebra.
\end{proposition}

\begin{proof}
The PBW theorem and \eqref{eq:cell-LS-fixed} give the
Ore presentation and the interval condition.
To prove local nilpotence, consider the functional
$\ell_k(\beta)=-\operatorname{ht}(w_{\le k-1}^{-1}\beta)$.
It is positive on $\beta_1,\ldots,\beta_{k-1}$ and has
value $-1$ on $\beta_k$.
Since $\delta_k$ raises root degree by $\beta_k$, it lowers
this nonnegative filtration by one and is locally nilpotent.
The displayed characters implement $\sigma_k$ and the
reverse Ore automorphisms, and $(\beta_k,\beta_k)=2d_{i_k}$.
Finally, integral straightening gives preservation by
$\delta_k$, while $\sigma_k^{\pm1}$ acts by powers of $q$.
\end{proof}

\begin{lemma}[Root coordinates and commutation exponents]
\label{lem:root-coordinate-matrices}
Identify \(Q\) with \(\mathbb Z^I\) using the simple roots.
Let \(S,G\in\operatorname{Mat}_{I\times[1,r]}(\mathbb Z)\)
have columns \(\alpha_{i_k}\) and \(\beta_k\), respectively.
Let \(T\) be upper unitriangular with
\(T_{kl}=c_{i_ki_l}\) for \(k<l\), and put
\[
E=T^{-1},\qquad
\Delta_{\mathbf w}=\operatorname{diag}(d_{i_1},\ldots,d_{i_r}),
\qquad \mathsf E=\Delta_{\mathbf w}E.
\]
Then
\[
\begin{aligned}
G&=SE,\\
S^{\mathsf t}D_CCS&=\Delta_{\mathbf w}T+T^{\mathsf t}\Delta_{\mathbf w},\\
G^{\mathsf t}D_CCG&=\Delta_{\mathbf w}E+E^{\mathsf t}\Delta_{\mathbf w}
=\mathsf E+\mathsf E^{\mathsf t}.
\end{aligned}
\]
In particular,
\(\beta_k=\sum_{j\leq k}E_{jk}\alpha_{i_j}\) and
\((\beta_k,\beta_l)=d_{i_k}E_{kl}=\mathsf E_{kl}\) for \(k<l\).
The skew-symmetric matrix \(\Lambda_0\) defined by
\(-(\Lambda_0)_{kl}=(\beta_k,\beta_l)\) for \(k<l\)
is therefore
\(\Lambda_0=\mathsf E^{\mathsf t}-\mathsf E
=E^{\mathsf t}\Delta_{\mathbf w}-\Delta_{\mathbf w}E\).
These identities also hold when the word is not reduced.
\end{lemma}

\begin{proof}
The identity
\[
1-w_{\leq k-1}
=\sum_{j=1}^{k-1}w_{\leq j-1}(1-s_{i_j})
\]
and the reflection formula
\((1-s_{i_j})\alpha_{i_k}=c_{i_ji_k}\alpha_{i_j}\)
give
\[
\alpha_{i_k}-\beta_k=\sum_{j<k}c_{i_ji_k}\beta_j.
\]
Thus \(GT=S\), and hence \(G=SE\).

The \((k,l)\)-entry of \(S^{\mathsf t}D_CCS\) is
\(d_{i_k}c_{i_ki_l}\). For \(k<l\), this is the
\((k,l)\)-entry of \(\Delta_{\mathbf w}T\).
For \(k>l\), the symmetrizing identity
\(d_{i_k}c_{i_ki_l}=d_{i_l}c_{i_li_k}\)
identifies it with the \((k,l)\)-entry of
\(T^{\mathsf t}\Delta_{\mathbf w}\).
The diagonal entries are \(2d_{i_k}\). Consequently,
\[
S^{\mathsf t}D_CCS=\Delta_{\mathbf w}T+T^{\mathsf t}\Delta_{\mathbf w}.
\]
Multiplying by \(E^{\mathsf t}\) and \(E\), and using \(TE=I\), yields
\[
\begin{aligned}
G^{\mathsf t}D_CCG
&=E^{\mathsf t}(\Delta_{\mathbf w}T+T^{\mathsf t}\Delta_{\mathbf w})E\\
&=E^{\mathsf t}\Delta_{\mathbf w}+\Delta_{\mathbf w}E
=\mathsf E^{\mathsf t}+\mathsf E.
\end{aligned}
\]
The formula for \(\beta_k\) follows from \(G=SE\).
Since \(\mathsf E\) is upper triangular, its \((k,l)\)-entry
for \(k<l\) is \((\beta_k,\beta_l)\), proving the assertion
about \(\Lambda_0\). The argument uses only the reflection
formula and the symmetrizer, so no reducedness assumption is needed.
\end{proof}

\subsection{The quantum cluster structure}

For \(k\in[1,r]\), let \(k^+\) and \(k^-\) be the next and
previous positions carrying \(i_k\), with values \(+\infty\)
and \(-\infty\), respectively, when no such position exists.
Set \(\operatorname{fr}(\mathbf w)=\{k\mid k^+=+\infty\}\)
and \(\operatorname{ex}(\mathbf w)=[1,r]\setminus
\operatorname{fr}(\mathbf w)\).
The extended exchange matrix
\(B_{\mathbf w}=(b_{st})_{s\in[1,r],\,t\in\operatorname{ex}(\mathbf w)}\)
is defined by
$$
b_{st}=
\begin{cases}
 -1, & s=t^+,\\
  1, & s^+=t,\\
  c_{i_si_t}, & s<t<s^+<t^+,\\
 -c_{i_si_t}, & t<s<t^+<s^+,\\
  0, & \text{otherwise}.
\end{cases}
$$
The principal part of \(B_{\mathbf w}\) is skew-symmetrizable,
with symmetrizer \(\operatorname{diag}(d_{i_k})_{k\in\operatorname{ex}(\mathbf w)}\).
Put \(\zeta_k=\varpi_{i_k}-w_{\leq k}\varpi_{i_k}\), and
define the skew-symmetric integer matrix \(\Lambda_{\mathbf w}\)
by \((\Lambda_{\mathbf w})_{ss}=0\) and, for \(s<t\),
$$
(\Lambda_{\mathbf w})_{st}
=(\zeta_s,\varpi_{i_t}+w_{\leq t}\varpi_{i_t}).
$$

For each \(k\in[1,r]\), define
\(\mathbf f_k=(f_{pk})_{p=1}^{r}\in\mathbb N^r\) by
\(f_{pk}=1\) if \(p\leq k\) and \(i_p=i_k\), and
\(f_{pk}=0\) otherwise.
Set \(D_k^{\mathbf w}=G^{\mathbf w}(\mathbf f_k)\).
These quantum minors satisfy
\(D_s^{\mathbf w}D_t^{\mathbf w}
=q^{(\Lambda_{\mathbf w})_{st}}D_t^{\mathbf w}D_s^{\mathbf w}\)
for all \(s,t\in[1,r]\).
To use the standard quantum mutation convention, set
\(X_k^{\mathbf w}=q^{-(\zeta_k,\zeta_k)/4}D_k^{\mathbf w}\);
this is the normalized-minor seed of
\cite[Section~7.1 and Theorem~7.3]{GY-integral}, in our
coordinate-ring conventions. In symmetric type it agrees with
\cite[Theorem~11.1.4]{kang2018monoidal}.
The compatibility identity is
\[
B_{\mathbf w}^{\mathsf t}\Lambda_{\mathbf w}
=2\bigl(d_{i_t}\delta_{tu}\bigr)_
 {t\in\operatorname{ex}(\mathbf w),\,u\in[1,r]}.
\]
The corresponding quantum seed is
\begin{equation}\label{eq:seed-w}
   \mathbf s(\mathbf w)= ((X_1^{\mathbf w},\ldots,X_r^{\mathbf w}),
 B_{\mathbf w},\Lambda_{\mathbf w}).
\end{equation}

The integral seed is defined over \(\A\).
Since \((\zeta_k,\zeta_k)\) is even, the normalization factor
is a unit of \(\A\).
For a singleton prime chain, \(D_k^{\mathbf w}=P_k^{\mathbf w}\)
and \((\zeta_k,\zeta_k)=2d_{i_k}\); its normalized seed variable is
therefore \(q_{i_k}^{-1/2}P_k^{\mathbf w}\).

\begin{theorem}[{\cite[Theorem~7.3]{GY-integral}}]\label{thm:cell-integral-cluster}
For every symmetrizable generalized Cartan matrix and reduced word
\(\mathbf w\) of \(w\), the normalized quantum minors identify
\[
\mathcal A_{\A}(\mathbf s(\mathbf w))
\cong A_{\A}(\mathfrak n(w)).
\]
No frozen variable is inverted. Scalar extension to \(\mathbb K\)
gives the corresponding isomorphism for \(A_q(\mathfrak n(w))\).
\end{theorem}

\section{Bosonic extension algebras}\label{sec:bosonic}

We recall the bosonic extension associated with the symmetrizable
generalized Cartan matrix \(C=(c_{ij})_{i,j\in I}\). We work over
\(\K=\mathbb Q(q^{1/2})\), where \(q\) is an indeterminate. We retain the symmetrizer
\(D_C=\operatorname{diag}(d_i)\), the parameters \(q_i=q^{d_i}\),
and the pairing \((\alpha_i,\alpha_j)=d_ic_{ij}\) fixed above.
All results in this section hold without a finite-type assumption.

\begin{definition}[Bosonic extension]
The bosonic extension \(\Ahat\) is the \(\K\)-algebra generated
by \(f_{i,m}\), for \(i\in I\) and \(m\in\mathbb Z\), subject to
the quantum Serre relations
$$
\sum_{a=0}^{1-c_{ij}}(-1)^a
\begin{bmatrix}1-c_{ij}\\a\end{bmatrix}_{q_i}
f_{i,m}^{\,1-c_{ij}-a}f_{j,m}f_{i,m}^{\,a}=0
\qquad(i\ne j,\ m\in\mathbb Z)
$$
and the relations
$$
f_{i,m}f_{j,n}
=q^{(-1)^{n-m+1}(\alpha_i,\alpha_j)}f_{j,n}f_{i,m}
+\delta_{ij}\delta_{n,m+1}(1-q_i^2)
\qquad(m<n).
$$
Here \(\begin{bmatrix}n\\a\end{bmatrix}_{q_i}=[n]_{q_i}!/([a]_{q_i}![n-a]_{q_i}!)\).
The algebra is \(Q\)-graded by
\(\deg f_{i,m}=(-1)^m\alpha_i\); we also write
\(\operatorname{wt}(u)=-\deg(u)\) for homogeneous \(u\).
\end{definition}

Set \(\mathcal F_{i,m}=q_i^{1/2}(1-q_i^2)^{-1}f_{i,m}\) and
\(\mathcal F_{i,m}^{(a)}=\mathcal F_{i,m}^a/[a]_{q_i}!\).
By \cite[Theorem~3.1]{KKOP-braid}, the formulas
$$
\T_i(f_{j,m})=
\begin{cases}
f_{i,m+1}, & j=i,\\[3pt]
\displaystyle
\sum_{a+b=-c_{ij}}(-q_i)^b
\mathcal F_{i,m}^{(a)}f_{j,m}\mathcal F_{i,m}^{(b)},
& j\ne i
\end{cases}
$$
define algebra automorphisms satisfying the braid relations.
For homogeneous \(u\), they satisfy
\(\deg(\T_i(u))=s_i(\deg(u))\).

Let \(\mathbf i=(i_1,\ldots,i_r)\) be a positive braid word
representing \(b\) in the positive braid monoid associated with
\(C\). Put \(w_{\leq0}=1\),
\(w_{\leq k}=s_{i_1}\cdots s_{i_k}\), and
\(\gamma_k=w_{\leq k-1}\alpha_{i_k}\).
The word \(\mathbf i\) need not be reduced in \(W\).

\begin{definition}[Bosonic PBW generators]\label{def:PBW-roots}
For \(k\in[1,r]\), define
\[P_k^{\mathbf i}=\T_{i_1}\cdots\T_{i_{k-1}}(q_{i_k}^{1/2}f_{i_k,0}).\]
Set
\(\Ahat(b)=\K\langle P_1^{\mathbf i},\ldots,P_r^{\mathbf i}\rangle
\subseteq\Ahat\).
Then \(\deg P_k^{\mathbf i}=\gamma_k\).
\end{definition}

The subalgebra \(\Ahat(b)\) is independent of the chosen
positive expression of \(b\), and the ordered monomials
\((P_r^{\mathbf i})^{a_r}\cdots(P_1^{\mathbf i})^{a_1}\),
with \((a_1,\ldots,a_r)\in\mathbb N^r\), form a \(\K\)-basis
\cite[Section~4.1 and Corollary~4.9]{KKOP-braid}.
For the fixed word \(\mathbf i\), write
\(\Ahat(b)[k,l]=\K\langle P_k^{\mathbf i},\ldots,P_l^{\mathbf i}\rangle\)
when \(k\leq l\), and \(\Ahat(b)[k,l]=\K\) when \(k>l\).
The Levendorskii--Soibelman relation
\cite[Proposition~4.13]{KKOP-braid} gives
\begin{equation}\label{eq:LS}
P_k^{\mathbf i}P_l^{\mathbf i}
-q^{-(\gamma_k,\gamma_l)}P_l^{\mathbf i}P_k^{\mathbf i}
\in\Ahat(b)[k+1,l-1]
\qquad(k<l).
\end{equation}

For the CGL presentation, write \(x_k=P_k^{\mathbf i}\) and set
\(\Ahat(b)[k,l]=\K\langle x_k,\ldots,x_l\rangle\) for
\(k\leq l\), and \(\Ahat(b)[k,l]=\K\) otherwise.
For \(k\geq2\) and homogeneous
\(u\in \Ahat(b)[1,k-1]\) of degree \(\beta\in Q\), define
\(\sigma_k(u)=q^{(\gamma_k,\beta)}u\) and
\(\delta_k(u)=x_ku-\sigma_k(u)x_k\), extending both maps linearly.

\begin{theorem}\label{thm:bosonic-symmetric-CGL}
For every positive braid word \(\mathbf i\), the algebra
\(\Ahat(b)\) has the symmetric CGL presentation
$$
\Ahat(b)
=\K[x_1][x_2;\sigma_2,\delta_2]\cdots
[x_r;\sigma_r,\delta_r].
$$
For the torus
\(\mathcal H=\operatorname{Hom}_{\mathbb Z}(Q,\K^\times)\)
acting through the \(Q\)-grading, one may take
\(h_k(\beta)=q^{(\gamma_k,\beta)}\) and
\(h_k^*(\beta)=q^{-(\gamma_k,\beta)}\).
Their eigenvalues on \(x_k\) are \(q_{i_k}^{2}\) and \(q_{i_k}^{-2}\),
respectively.
\end{theorem}

\begin{proof}
Write \(R_k=\Ahat(b)[1,k]\) and \(R_0=\K\).
The PBW theorem, applied to each prefix of \(\mathbf i\),
shows that the monomials \(x_1^{a_1}\cdots x_k^{a_k}\)
form a \(\K\)-basis of \(R_k\).
Since the \(x_j\) are homogeneous, \(\sigma_k\) is an
automorphism of \(R_{k-1}\).
Rearranging \eqref{eq:LS} gives
$$
\delta_k(x_j)
=x_kx_j-q^{(\gamma_k,\gamma_j)}x_jx_k
\in \Ahat(b)[j+1,k-1]
\qquad(j<k).
$$
By definition, \(\delta_k\) satisfies the skew Leibniz rule
\(\delta_k(uv)=\sigma_k(u)\delta_k(v)+\delta_k(u)v\).
The preceding inclusion therefore implies that
\(\delta_k(R_{k-1})\subseteq R_{k-1}\).
The natural map
\(R_{k-1}[X;\sigma_k,\delta_k]\to R_k\), sending \(X\) to
\(x_k\), is surjective. It is injective by the PBW basis,
so \(R_k=R_{k-1}[x_k;\sigma_k,\delta_k]\).

Since \eqref{eq:LS} gives
\(\delta_k(x_j)\in \Ahat(b)[j+1,k-1]\) for every \(j<k\),
we can assign positive integer weights to \(x_{k-1},x_{k-2},\ldots,x_1\),
in that order, so that every word in a fixed expression for
\(\delta_k(x_j)\) has smaller weight than \(x_j\).
Since \(\sigma_k\) preserves the resulting word filtration,
the skew Leibniz rule shows that \(\delta_k\) lowers it strictly. Thus \(\delta_k\) is locally
nilpotent.

Finally, the \(Q\)-grading induces a rational
\(\mathcal H\)-action for which each \(x_k\) is an eigenvector,
and \(h_k|_{R_{k-1}}=\sigma_k\).
As \(\gamma_k\) is a Weyl group translate of a simple root,
\((\gamma_k,\gamma_k)=2d_{i_k}\), so \(h_k\cdot x_k=q_{i_k}^{2}x_k\).
As this eigenvalue is not a root of unity, the CGL conditions hold.
The interval inclusion for \(\delta_k(x_j)\) gives the first
symmetry condition. For \(l>j\),
\(h_j^*\cdot x_l=q^{-(\gamma_j,\gamma_l)}x_l
=\lambda_{lj}^{-1}x_l\), where
\(\lambda_{lj}=q^{(\gamma_l,\gamma_j)}\), and
\(h_j^*\cdot x_j=q_{i_j}^{-2}x_j\).
This proves the second symmetry condition.
\end{proof}

\subsection{Global basis and the standard integral form}

Let \(B(\infty)\) be the crystal of \(U_q^-(\mathfrak g)\),
with highest element \(u_\infty\), and let
\(G^{\mathrm{up}}(x)\) denote the corresponding upper global
basis element of \(A_q(\mathfrak n)\).
We use the usual crystal operators
\(\widetilde e_i,\widetilde f_i\) and functions
\(\varepsilon_i,\varphi_i\).
Let \(*\) denote Kashiwara's involution, and set
\(\widetilde e_i^*=*\widetilde e_i*\),
\(\widetilde f_i^*=*\widetilde f_i*\),
\(\varepsilon_i^*(x)=\varepsilon_i(x^*)\), and
\(\varphi_i^*(x)=\varepsilon_i^*(x)+
\langle h_i,\wt x\rangle\).

The extended crystal is
$$
\widehat B(\infty)
=
\left\{
\mathbf b=(b_m)_{m\in\mathbb Z}
\;\middle|\;
b_m\in B(\infty),\
b_m=u_\infty \text{ for all but finitely many }m
\right\}.
$$
For \(m\in\mathbb Z\), let
\(\Ahat[m]=\mathbb K\langle f_{i,m}\mid i\in I\rangle\).
For an interval $[u,v]$, set
$\Ahat[u,v]=\mathbb K\langle f_{i,m}\mid i\in I;m\in [u,v]\rangle$.
After extending scalars to \(\mathbb K\), let
\(\psi_m:A_q(\mathfrak n)\to\Ahat[m]\) be the standard
isomorphism normalized by
\(\psi_m(G^{\mathrm{up}}(\widetilde f_i u_\infty))
=q_i^{1/2}f_{i,m}\).
Write \(g_m(x)=\psi_m(G^{\mathrm{up}}(x))\).
For \(\mathbf b\) supported in an interval \([u,v]\), define
$$
P(\mathbf b)=g_v(b_v)\cdots g_u(b_u),
\qquad
\beta_m(\mathbf b)=-\wt b_m\in Q_+.
$$
This product is independent of the choice of \([u,v]\).
Define \(\mathbf c\preceq\mathbf b\) if
\(\beta_m(\mathbf c)\le\beta_m(\mathbf b)\) for every \(m\),
where \(\mu\le\nu\) means \(\nu-\mu\in Q_+\).
Write \(\mathbf c\prec\mathbf b\) if, in addition, at least
one of these inequalities is strict.

The bosonic global basis
\(\{G(\mathbf b)\mid\mathbf b\in\widehat B(\infty)\}\)
constructed in \cite[Theorem~7.6]{KKOP-global} satisfies
\begin{equation}\label{eq:global-pbw}
G(\mathbf b)-P(\mathbf b)
\in
\sum_{\mathbf c\prec\mathbf b}q\mathbb Z[q]P(\mathbf c).
\end{equation}
For every positive braid \(b\), its intersection with
\(\Ahat(b)\) is a basis of \(\Ahat(b)\)
\cite[Corollary~4.11]{KKOP-braid}.

For \(\mathbf a=(a_1,\ldots,a_r)\in\mathbb N^r\), put
\[
P^{\mathbf i}(\mathbf a)
=\left(\prod_{k=1}^r q_{i_k}^{\binom{a_k}{2}}\right)
(P_r^{\mathbf i})^{a_r}\cdots(P_1^{\mathbf i})^{a_1}.
\]
Define the \emph{standard integral form} by
\begin{equation}\label{eq:bosonic-integral-form}
\Ahat_{\A}(b)
:=\bigoplus_{G\in\mathbf G\cap\Ahat(b)}\A G
=\bigoplus_{\mathbf a\in\mathbb N^r}\A P^{\mathbf i}(\mathbf a),
\end{equation}
where \(\mathbf G\) denotes the bosonic global basis.
The equality of these spans, closure under multiplication, and
independence of \(\mathbf i\) follow from
\cite[Corollaries~4.9 and~4.11]{KKOP-braid}, after extending
\(\mathbb Z[q^{\pm1}]\) to \(\A\).
The integral straightening law
\cite[Proposition~4.13]{KKOP-braid} has coefficients in
\(\mathbb Z[q^{\pm1}]\). Thus the ordinary ordered monomials
in the \(P_k^{\mathbf i}\), in either ascending or descending order,
are \(\A\)-bases, and
\(\mathbb K\otimes_{\A}\Ahat_{\A}(b)=\Ahat(b)\).
The braid symmetries preserve \(\mathbf G\), and hence the
integral form whenever they identify the corresponding subalgebras.
Their explicit formulas are evaluated over $\mathbb Q(q^{1/2})$;
individual terms need not be integral.

The root-torus maps \(h_k,h_k^*\) and the Ore maps also preserve
the integral prefix algebras. We apply the classification of
homogeneous prime elements and the invariant-center arguments
over \(\mathbb K\).

Let $\widehat A_{\A}$ be the full bosonic integral form spanned by
$\mathbf G$. 
For a position interval $[a,d]\subseteq[1,r]$, let
$\widehat A_{\A}(b)[a,d]$ denote the $\A$-span of the ordered
PBW monomials supported in $[a,d]$; for an empty position interval,
set $\widehat A_{\A}(b)[a,d]=\A$.

\begin{lemma}[Integral interval subalgebras]
\label{lem:integral-interval-saturation}
For every position interval $[a,d]$, one has
\begin{equation}\label{eq:integral-interval-intersection}
 \widehat A_{\A}\cap\widehat A(b)[a,d]
 =\widehat A_{\A}(b)[a,d].
\end{equation}
This integral interval algebra is saturated in $\widehat A_{\A}$
and is a direct summand, as an $\A$-module, of every integral
prefix algebra $\widehat A_{\A}(b)[1,n]$ with $d\le n\le r$.
\end{lemma}
\begin{proof}
For an empty interval the assertion is immediate. Otherwise, put
$\mathbf j=(i_a,\ldots,i_d)$,
$b_{[a,d]}=\sigma_{i_a}\cdots\sigma_{i_d}$, and
$\T_{<a}=\T_{i_1}\cdots\T_{i_{a-1}}$.
The definitions of the PBW generators give
\[
 \widehat A(b)[a,d]
 =\T_{<a}\bigl(\widehat A(b_{[a,d]})\bigr),
 \qquad
 \widehat A_{\A}(b)[a,d]
 =\T_{<a}\bigl(\widehat A_{\A}(b_{[a,d]})\bigr).
\]
By \cite[Corollary~4.11]{KKOP-braid},
$\widehat A(b_{[a,d]})$ is spanned over $\mathbb K$ by a subset
of $\mathbf G$, and its integral form is the $\A$-span of the
same subset. Thus its integral form is the intersection of this
subalgebra with $\widehat A_{\A}$.
The automorphism $\T_{<a}$ permutes $\mathbf G$
\cite[Theorem~3.7]{KKOP-braid}; applying it proves
\eqref{eq:integral-interval-intersection}.

Suppose that $0\ne s\in\A$, $u\in\widehat A_\A$, and
$su\in\widehat A_\A(b)[a,d]$.
Since $s$ is invertible in $\mathbb K$, we obtain
$u=s^{-1}(su)\in\widehat A(b)[a,d]$.
Since $u\in\widehat A_\A$,
\eqref{eq:integral-interval-intersection} gives
$u\in\widehat A_\A(b)[a,d]$.
Thus the interval integral form is saturated.

In each prefix containing $[a,d]$, the PBW monomials supported
in this interval form a subset of the ordered PBW basis.
The remaining monomials span an $\A$-module complement.
\end{proof}

To describe the braid action on the global basis, recall the
Saito reflection
\(\mathscr S_j\), defined on \(\varepsilon_j(x)=0\), and
its extension:
$$
\mathscr S_j(x)
=
\widetilde f_j^{\,\varphi_j^*(x)}
(\widetilde e_j^*)^{\varepsilon_j^*(x)}x,
\qquad
\widetilde{\mathscr S}_j(x)
=
\mathscr S_j\bigl(
\widetilde e_j^{\,\varepsilon_j(x)}x\bigr).
$$
By \cite[Theorem~3.7]{KKOP-braid},
$$
\T_jG(\mathbf b)=G(\mathsf R_j\mathbf b),
\qquad
(\mathsf R_j\mathbf b)_m
=
(\widetilde f_j^*)^{\varepsilon_j(b_{m-1})}
\widetilde{\mathscr S}_j(b_m).
$$

\begin{lemma}\label{lem:star-epsilon-monotonicity}
Let \(i\ne j\) and \(x\in B(\infty)\).
If \(\widetilde e_j^*x\ne0\), then
\(\varepsilon_i^*(\widetilde e_j^*x)
\ge\varepsilon_i^*(x)\).
Moreover,
\begin{equation}\label{eq:epsilon-monotonicity}
\varepsilon_i^*(\widetilde{\mathscr S}_j x)
\ge\varepsilon_i^*(x).
\end{equation}
\end{lemma}

\begin{proof}
By \cite[Proposition~2.4 and Theorems~6.21,~7.4]
{LaudaVazirani2011},
\(\varepsilon_i(\widetilde f_jy)\le\varepsilon_i(y)\)
for \(i\ne j\). These results concern arbitrary symmetrizable
Cartan data: one first chooses a sufficiently large dominant weight
using Proposition~2.4, applies Theorem~6.21, and then uses
Theorem~7.4 to identify the resulting crystal with \(B(\infty)\).
Taking \(y=\widetilde e_j(x^*)\), we obtain
$$
\varepsilon_i^*(x)
=\varepsilon_i(\widetilde f_jy)
\le\varepsilon_i(y)
=\varepsilon_i^*(\widetilde e_j^*x).
$$

For \(i\ne j\), we also have
\(\varepsilon_i^*(\widetilde f_jz)=\varepsilon_i^*(z)\)
by \cite[Proposition~6.2(ii)]{LaudaVazirani2011}.
Since \(\widetilde f_j\widetilde e_jz=z\), the same
invariance holds under \(\widetilde e_j\) whenever
\(\widetilde e_jz\ne0\).
Set \(c=\widetilde e_j^{\varepsilon_j(x)}x\),
\(a=\varepsilon_j^*(c)\), and \(d=\varphi_j^*(c)\).
Since \(\varepsilon_j(c)=0\), the Saito reflection is defined
on \(c\) and \(d\geq0\).
Then
$$
\begin{aligned}
\varepsilon_i^*(\widetilde{\mathscr S}_j x)
&=\varepsilon_i^*\bigl(
  \widetilde f_j^{\,d}(\widetilde e_j^*)^a c\bigr)\\
&=\varepsilon_i^*\bigl((\widetilde e_j^*)^a c\bigr)
\ge\varepsilon_i^*(c)
=\varepsilon_i^*(x),
\end{aligned}
$$
where the inequality follows by applying the first
assertion \(a\) times.
\end{proof}

For homogeneous \(x,y\in\Ahat\), write
\([x,y]_q=xy-q^{-(\wt x,\wt y)}yx\), and set
\(\mathcal E_i^*(x)=[f_{i,-1},x]_q\).

\begin{lemma}[Nonvanishing of straightening corrections]
\label{lem:nonvanishing}
Let \(\mathbf j=(j_1,\ldots,j_t)\) be a word with
\(j_a\ne i\) for all \(a\). Then
$$
\bigl[f_{i,-1},
\T_{j_1}\cdots\T_{j_t}(f_{i,0})\bigr]_q\ne0.
$$
Consequently, the straightening correction between the
first and last bosonic PBW generators of
\((i,j_1,\ldots,j_t,i)\) is nonzero.
\end{lemma}

\begin{proof}
Let \(\mathbf b^{(i)}\) have zeroth component
\(\widetilde f_i u_\infty\) and all other components
\(u_\infty\), and put
\(\mathbf b=\mathsf R_{j_1}\cdots\mathsf R_{j_t}
\mathbf b^{(i)}\).
The formula for \(\mathsf R_j\) and
\eqref{eq:epsilon-monotonicity} give
$$
b_m=u_\infty\quad(m<0),
\qquad
b_0=\widetilde{\mathscr S}_{j_1}\cdots
\widetilde{\mathscr S}_{j_t}
(\widetilde f_i u_\infty),
\qquad
\varepsilon_i^*(b_0)\ge1.
$$
Thus
\(X:=\T_{j_1}\cdots\T_{j_t}(q_i^{1/2}f_{i,0})
=G(\mathbf b)\), with \(\mathbf b\) supported in
\([0,M]\) for some \(M\ge0\).

By \cite[Lemma~6.1]{KKOP-global}, for homogeneous \(u,v\),
$$
\mathcal E_i^*(uv)
=
\mathcal E_i^*(u)v
+q^{-(\alpha_i,\wt u)}u\mathcal E_i^*(v),
\qquad
\mathcal E_i^*(\Ahat[m])=0\quad(m>0).
$$
Compatibility of the upper global basis with the right skew
derivative gives \cite[Section~1.4]{KKOP-braid}
$$
\ker\bigl(\mathcal E_i^*|_{\Ahat[0]}\bigr)
=
\bigoplus_{\varepsilon_i^*(c)=0}\mathbb K g_0(c).
$$
In particular,
\(\mathcal E_i^*(g_0(b_0))\ne0\).

For sequences supported in \([0,M]\), write
\(\boldsymbol\beta(\mathbf c)
=(\beta_0(\mathbf c),\ldots,\beta_M(\mathbf c))\).
Let \(V_{\boldsymbol\nu}\) be the span of the elements
\(P(\mathbf c)\) with
\(\boldsymbol\beta(\mathbf c)=\boldsymbol\nu\).
The ordered tensor-product decomposition gives
\(\Ahat[0,M]=\bigoplus_{\boldsymbol\nu}V_{\boldsymbol\nu}\),
where \(\Ahat[0,M]\) is the subalgebra generated by
layers \(0,\ldots,M\).
Denote the corresponding projections by
\(\operatorname{pr}_{\boldsymbol\nu}\), and set
\(\boldsymbol a_i=(\alpha_i,0,\ldots,0)\).
For \(Y_{\mathbf c}=g_M(c_M)\cdots g_1(c_1)\),
the skew Leibniz rule yields
$$
\mathcal E_i^*(P(\mathbf c))
=
q^{-(\alpha_i,\wt Y_{\mathbf c})}
Y_{\mathbf c}\mathcal E_i^*(g_0(c_0))
\in V_{\boldsymbol\beta(\mathbf c)-\boldsymbol a_i}.
$$
Here \(V_{\boldsymbol\nu}=0\) if any component of
\(\boldsymbol\nu\) lies outside \(Q_+\).
With \(\boldsymbol\beta=\boldsymbol\beta(\mathbf b)\),
the componentwise root order gives
$$
\mathbf c\prec\mathbf b
\ \Longrightarrow\
\boldsymbol\beta(\mathbf c)-\boldsymbol a_i
<
\boldsymbol\beta-\boldsymbol a_i
\ \Longrightarrow\
\operatorname{pr}_{\boldsymbol\beta-\boldsymbol a_i}
\mathcal E_i^*(P(\mathbf c))=0.
$$
Applying this projection to \eqref{eq:global-pbw}, we obtain
\begin{equation}\label{eq:leading-layer-derivative}
\begin{aligned}
\operatorname{pr}_{\boldsymbol\beta-\boldsymbol a_i}
\mathcal E_i^*(X)
&=\mathcal E_i^*(P(\mathbf b))\\
&=q^{-(\alpha_i,\wt Y_{\mathbf b})}
Y_{\mathbf b}\mathcal E_i^*(g_0(b_0))
\ne0.
\end{aligned}
\end{equation}
The final inequality follows from
\(\mathcal E_i^*(g_0(b_0))\ne0\) and the injectivity of
ordered multiplication
\(\Ahat[M]\otimes\cdots\otimes\Ahat[0]\to\Ahat[0,M]\).
Hence
\([f_{i,-1},\T_{j_1}\cdots\T_{j_t}(f_{i,0})]_q
=q_i^{-1/2}\mathcal E_i^*(X)\ne0\).

Finally, let \(P_1,\ldots,P_{t+2}\) be the PBW generators
of \((i,j_1,\ldots,j_t,i)\), and set
\(\gamma=s_{j_1}\cdots s_{j_t}\alpha_i\).
Then \(P_1=q_i^{1/2}f_{i,0}\),
\(P_{t+2}=\T_i(X)\), and \(\wt X=-\gamma\).
Using \(\T_i^{-1}(f_{i,0})=f_{i,-1}\) and
\((\alpha_i,s_i\gamma)=-(\alpha_i,\gamma)\), we obtain
$$
\begin{aligned}
&\T_i^{-1}\!\left(
P_1P_{t+2}
-q^{-(\alpha_i,s_i\gamma)}P_{t+2}P_1
\right)\\
&\qquad
=q_i^{1/2}\left(
f_{i,-1}X-q^{(\alpha_i,\gamma)}Xf_{i,-1}
\right)
=q_i^{1/2}\mathcal E_i^*(X)\ne0.
\end{aligned}
$$
The straightening correction is therefore nonzero, since
\(\T_i\) is an automorphism.
\end{proof}

Put $\zeta_i=1-q_i^2$ for $i\in I$. For
$\beta=\sum_{j\in I}n_j\alpha_j\in Q$, set
\begin{equation}
\zeta^\beta:=\prod_{j\in I}\zeta_j^{n_j}\in\mathbb K^\times.    
\end{equation}

Let $F$ be a free $\A$-module with basis
$\{e_j\}_{j\in J}$. Write each $u\in F$ uniquely as
\(
u=\sum_{j\in J}u_je_j\) for \(
 u_j\in\A,
\)
where only finitely many coordinates $u_j$ are nonzero.
We call a nonzero element $u$ \emph{primitive} if its coordinates
have no common nonunit divisor in $\A$.
Equivalently, $u$ is primitive if every factorization
$u=sw$, with $s\in\A$ and $w\in F$, forces $s$ to be a unit.
Since such a factorization exists exactly when $s$ divides
every coordinate, primitivity is independent of the basis.
Recall that $\A=\mathbb Z[q^{\pm1/2}]$ is a unique factorization
domain with fraction field $\mathbb K=\mathbb Q(q^{1/2})$, and
\(
\A^\times=\{\pm q^{m/2}\mid m\in\mathbb Z\}.
\)

\begin{lemma}\label{lem:primitive-rescaling}
Let $F$ be a free $\A$-module, and let $u,v\in F$ be nonzero
primitive elements. If $u=cv$ in $\mathbb K\otimes_\A F$
for some $c\in\mathbb K^\times$, then $c\in\A^\times$.
\end{lemma}

\begin{proof}
Choose an $\A$-basis $\{e_j\}_{j\in J}$ of $F$ and write
\(
u=\sum_j u_je_j\) and \(v=\sum_j v_je_j.
\)
Since $\A$ is a unique factorization domain, we may write
$c=a/b$, where $a,b\in\A\setminus\{0\}$ have no common
irreducible factor.
The equality $u=cv$ gives
\(
bu_j=av_j\)
for every $j$.

Since $a$ and $b$ have no common irreducible factor, $b$
divides every $v_j$.
Hence
\(
v=bw\),\(
w:=\sum_j (v_j/b)e_j\in F.
\)
As $v$ is primitive, $b$ must be a unit. Therefore $c\in\A$.

Now $u=cv$ is a factorization in $F$, so primitivity of $u$
forces $c$ to be a unit.
\end{proof}
Let $x\in U_q^-(\mathfrak g)$ be homogeneous of weight $-\beta$,
and set $a=\iota(x)$. By \cite[(1.10)]{KKOP-braid}, the natural
pairing and Kashiwara's bilinear form satisfy
\begin{equation}\label{eq:coordinate-kashiwara-pairing}
\langle a,y\rangle
=\zeta^{-\beta}(x,y)_K
\qquad\text{for every }y\in U_q^-(\mathfrak g).
\end{equation}
Write $G_K^{\mathrm{up}}(b)$ for the upper global basis of
$U_q^-(\mathfrak g)$ defined using Kashiwara's form, and retain
$G^{\mathrm{up}}(b)$ for its coordinate-ring version. If
$\beta=-\operatorname{wt}(b)\in Q_+$, the conventions
\cite[(1.8)--(1.10)]{KKOP-braid} give
\[
 G^{\mathrm{up}}(b)
 =\zeta^\beta\iota\bigl(G_K^{\mathrm{up}}(b)\bigr).
\]
Let $e_i^*$ be the right skew derivative on $U_q^-(\mathfrak g)$,
adjoint to right multiplication by $f_i$ for Kashiwara's form, and set
\[
 \partial_i^*
 :=\zeta_i^{-1}\iota\circ e_i^*\circ\iota^{-1}.
\]
This operator is initially defined over $\mathbb K$.

\begin{lemma}\label{lem:integral-skew-derivative}
The operator $\partial_i^*$ preserves $A_{\A}(\mathfrak n)$.
If $\varepsilon_i^*(b)=1$, then
\[
 \partial_i^*G^{\mathrm{up}}(b)
 =G^{\mathrm{up}}(\widetilde e_i^*b).
\]
Consequently, with
$g_0(b)=\psi_0(G^{\mathrm{up}}(b))$ and
$\mathcal E_i^*(u)=[f_{i,-1},u]_q$, one has
\[
 \frac{\mathcal E_i^*(g_0(b))}{1-q_i^2}
 =q^{(\alpha_i,\beta)}q_i^{-3/2}
   g_0(\widetilde e_i^*b),
 \qquad \beta=-\operatorname{wt}(b),\quad
 \varepsilon_i^*(b)=1.
\]
Moreover, $\zeta_i^{-1}\mathcal E_i^*$ preserves
$\widehat A_{\A}[0,M]$ for every $M\ge0$.
\end{lemma}
\begin{proof}
To prove the first assertion, take a homogeneous
$a=\iota(x)\in A_\A(\mathfrak n)$, where
$\operatorname{wt}(x)=-\beta$.
Let $U_\A^-(\mathfrak g)$ denote the divided-power integral
form, and take a homogeneous element $v\in U_\A^-(\mathfrak g)$.
The relevant pairings vanish unless
$\operatorname{wt}(v)=-\beta+\alpha_i$.
For this weight, \eqref{eq:coordinate-kashiwara-pairing}
and the adjoint property give
\[
\begin{aligned}
\langle\partial_i^*a,v\rangle
&=\zeta_i^{-1}\langle\iota(e_i^*x),v\rangle\\
&=\zeta_i^{-1}\zeta^{-(\beta-\alpha_i)}(e_i^*x,v)_K\\
&=\zeta^{-\beta}(x,vf_i)_K
=\langle a,vf_i\rangle.
\end{aligned}
\]
Here $\zeta^{\alpha_i}=\zeta_i$. Since multiplication by $f_i$
preserves the divided-power integral form, the last pairing lies
in $\A$. This proves $\partial_i^*a\in A_\A(\mathfrak n)$;
compare the left-handed statement in
\cite[Lemma~4.2]{KKOP-global}.

For completeness, the upper-global-basis string formula in arbitrary
symmetrizable type, in the right-handed convention, is, for
$0\le n\le\varepsilon_i^*(b)$,
\[
 \frac{(e_i^*)^n}{[n]_{q_i}!}G_K^{\mathrm{up}}(b)
 =\begin{bmatrix}\varepsilon_i^*(b)\\n\end{bmatrix}_{q_i}
   G_K^{\mathrm{up}}((\widetilde e_i^*)^n b)
 +\sum_{\varepsilon_i^*(b')<\varepsilon_i^*(b)-n}
   a_{b'}G_K^{\mathrm{up}}(b').
\]
This is \cite[Theorem~3.14 and Corollary~3.16]{Kimura-unipotent}.
When $\varepsilon_i^*(b)=n=1$, the sum is empty and the
binomial coefficient is $1$. Hence
$e_i^*G_K^{\mathrm{up}}(b)=G_K^{\mathrm{up}}(\widetilde e_i^*b)$.
Writing $\beta=-\operatorname{wt}(b)$, we obtain
\[
\begin{aligned}
\partial_i^*G^{\mathrm{up}}(b)
&=\zeta_i^{-1}\zeta^\beta
  \iota\bigl(e_i^*G_K^{\mathrm{up}}(b)\bigr)\\
&=\zeta^{\beta-\alpha_i}
  \iota\bigl(G_K^{\mathrm{up}}(\widetilde e_i^*b)\bigr)
=G^{\mathrm{up}}(\widetilde e_i^*b).
\end{aligned}
\]

For homogeneous $x\in U_q^-(\mathfrak g)$ of weight $-\beta$,
\cite[Lemma~6.1(iii)]{KKOP-global} gives
\[
\mathcal E_i^*(\psi_0(\iota(x)))
=q_i^{-3/2+\langle h_i,\beta\rangle}
 \psi_0(\iota(e_i^*x))
=q^{(\alpha_i,\beta)}q_i^{-3/2}
 \psi_0(\iota(e_i^*x)).
\]
Since $\iota(e_i^*x)=\zeta_i\partial_i^*(\iota(x))$,
every homogeneous coordinate element $a$ of degree $\beta$ satisfies
\[
\zeta_i^{-1}\mathcal E_i^*(\psi_0(a))
=q^{(\alpha_i,\beta)}q_i^{-3/2}\psi_0(\partial_i^*a).
\]
The scalar belongs to $\A^\times$, since $d_i$ and
$(\alpha_i,\beta)$ are integers. This proves the claimed
formula and preservation of $\widehat A_\A[0]$.

Finally, the integral ordered tensor-product decomposition expresses
$\widehat A_\A[0,M]$ as the $\A$-span of products
$x_M\cdots x_0$ with homogeneous $x_k\in\widehat A_\A[k]$.
By \cite[Lemma~6.1(ii)]{KKOP-global},
$\mathcal E_i^*(x_k)=0$ for $k>0$.
The skew Leibniz rule therefore gives
\[
\zeta_i^{-1}\mathcal E_i^*(x_M\cdots x_0)
=q^{-(\alpha_i,\operatorname{wt}(x_M\cdots x_1))}
 x_M\cdots x_1\bigl(\zeta_i^{-1}\mathcal E_i^*(x_0)\bigr).
\]
The scalar is again a unit of $\A$, and the last factor belongs
to $\widehat A_\A[0]$. This proves the final assertion.
\end{proof}

\begin{lemma}[Content of a straightening correction]
\label{lem:primitive-LS}
Let $p<k$ be positions in a positive word $\mathbf i$ with
$p=k^-$. Then
\[
 C_{p,k}:=P_p^{\mathbf i}P_k^{\mathbf i}
 -q^{-(\gamma_p,\gamma_k)}P_k^{\mathbf i}P_p^{\mathbf i}
 \in(1-q_{i_k}^2)\widehat A_{\A}(b)[p+1,k-1],
\]
and $C_{p,k}/(1-q_{i_k}^2)$ is primitive in this integral interval
algebra, whose $\A$-basis consists of the ordered PBW monomials.
\end{lemma}
\begin{proof}
Consider first the word $(i,j_1,\ldots,j_t,i)$ with every $j_a\ne i$,
and put
\[
 X=\mathbf T_{j_1}\cdots\mathbf T_{j_t}(q_i^{1/2}f_{i,0})
   =G(\mathbf b).
\]
As in Lemma~\ref{lem:nonvanishing}, the extended-crystal formula
gives $b_m=u_\infty$ for $m<0$ and
$\varepsilon_i^*(b_0)\ge1$.
Every operator used in $\mathsf R_j$, for $j\ne i$, changes a
component weight only by a multiple of $\alpha_j$.
It follows that the coefficient of $\alpha_i$ in
$\beta_0(\mathbf b)=-\operatorname{wt}(b_0)$ is exactly one,
and its coefficient in every other $\beta_m(\mathbf b)$ is zero.
Hence $\varepsilon_i^*(b_0)=1$.

Choose $M\ge0$ containing the support of $\mathbf b$.
Lemma~\ref{lem:integral-skew-derivative} gives
$\zeta_i^{-1}\mathcal E_i^*(X)\in\widehat A_{\A}[0,M]$.
Applying the same lemma to $b_0$ in
\eqref{eq:leading-layer-derivative} yields
\[
 \operatorname{pr}_{\boldsymbol\beta(\mathbf b)
                       -(\alpha_i,0,\ldots,0)}
 \bigl(\zeta_i^{-1}\mathcal E_i^*(X)\bigr)
 =q^{N/2}g_M(b_M)\cdots g_1(b_1)
       g_0(\widetilde e_i^*b_0)
\]
for some integer $N$. This projection is a single ordered
layer-product basis vector with a unit coefficient. The products
$P(\mathbf c)$, for $\mathbf c\in\widehat B(\infty)$, form an
$\A$-basis of the full bosonic integral form; hence
$\zeta_i^{-1}\mathcal E_i^*(X)$ is primitive there.

Applying $\mathbf T_i^{-1}$ to the straightening correction
between the endpoints gives
$q_i^{1/2}\mathcal E_i^*(X)$, exactly as in
Lemma~\ref{lem:nonvanishing}.  The automorphism $\mathbf T_i$
preserves the integral form and therefore preserves primitivity.
The same holds for the common braid prefix preceding position $p=k^-$.
Consequently $C_{p,k}/\zeta_{i_k}$ is integral and primitive in
$\widehat A_{\A}$ for every pair $p<k$ with $p=k^-$.

The field-level LS relation places this quotient in
$\widehat A(b)[p+1,k-1]$.
Lemma~\ref{lem:integral-interval-saturation} therefore gives
$C_{p,k}/\zeta_{i_k}\in\widehat A_{\A}(b)[p+1,k-1]$.
A nonunit divisor in the interval lattice would also divide
this element in $\widehat A_{\A}$, contradicting its primitivity
there. Thus it is primitive in the interval lattice as well.
\end{proof}

\section{Thickening and integral comparison}\label{sec:new-subalgebra}

\subsection{Thickening and dehomogenization}

Let \(\mathbf i=(i_1,\ldots,i_r)\) be a word in \(I\),
with \(r\ge1\), and set
\(\widetilde I=I\sqcup\{\infty_1,\ldots,\infty_{r-1}\}\).
Extend \(C\) to a generalized Cartan matrix \(\widetilde C\)
by setting, for \(i\in I\) and distinct new vertices \(a,b\),
\[
\widetilde c_{aa}=2,\qquad
\widetilde c_{ia}=-2,\qquad
\widetilde c_{ai}=-2d_i,\qquad
\widetilde c_{ab}=-2.
\]
Put \(\widetilde d_i=d_i\) for \(i\in I\) and
\(\widetilde d_a=1\) at every new vertex. Then
\(\widetilde D_C\widetilde C\) is symmetric, where
\(\widetilde D_C=\operatorname{diag}(\widetilde d_a)_{a\in\widetilde I}\).
Thus \(\widetilde C\) is symmetrizable and extends the fixed
Cartan datum.
Let \(\widetilde W\) and
\(\widetilde Q=\bigoplus_{a\in\widetilde I}\mathbb Z\alpha_a\)
be the corresponding Weyl group and root lattice.
We extend the bilinear form by
\((\alpha_a,\alpha_b)=\widetilde d_a\widetilde c_{ab}\), and identify
\(W\) with the standard parabolic subgroup of
\(\widetilde W\), so that \(s_i=\widetilde s_i\) for \(i\in I\).

Define the thickened word
$$
\mathbf j=\operatorname{th}(\mathbf i)
=(i_1,\infty_1,i_2,\ldots,\infty_{r-1},i_r).
$$
For an old vertex \(i\) and a new vertex \(a\), the product
\(\widetilde c_{ia}\widetilde c_{ai}=4d_i\) is at least four.
Every adjacent pair in \(\mathbf j\) therefore has infinite
Coxeter exponent, so no finite braid move or cancellation applies. The Coxeter word
reduction theorem therefore shows that \(\mathbf j\) is reduced.
Write
\(\widetilde w_{\le s}
=\widetilde s_{j_1}\cdots\widetilde s_{j_s}\),
with \(\widetilde w_{\le0}=1\), and set
\(\widetilde w_{\mathbf i}=\widetilde w_{\le2r-1}\) and
\(A_{\mathbf i}
=A_q(\mathfrak n(\widetilde w_{\mathbf i}))\),
considered over \(\mathbb K\).

Let \(X_1,\ldots,X_{2r-1}\) be the dual PBW generators
associated with \(\mathbf j\), with
\(\deg X_s=\widetilde\beta_s
=\widetilde w_{\le s-1}\alpha_{j_s}\) and
\(\wt X_s=-\widetilde\beta_s\).
With the convention of Section~\ref{sec:CGL}, we have
\begin{equation}\label{eq:quantum-cell-LS}
X_sX_t-q^{-(\widetilde\beta_s,\widetilde\beta_t)}X_tX_s
\in\mathbb K\langle X_{s+1},\ldots,X_{t-1}\rangle
\qquad(s<t).
\end{equation}
The algebra on the right is \(\mathbb K\) if \(t=s+1\).

Set \(H_\ell=q^{-1/2}X_{2\ell}\) for \(1\le\ell<r\).
Since \(\widetilde d_{\infty_\ell}=1\), this normalization makes
the separators ambient frozen cluster variables. Multiplication by an \(\A\)-unit leaves their degrees,
normality, and localization unchanged.
The CGL prime chains are indexed by the letters of
\(\mathbf j\). Since each \(\infty_\ell\) occurs exactly
once, \(H_\ell\) is a homogeneous normal prime element
\cite[Theorems~3.2 and~7.1]{GY-integral}.
The \(H_\ell\) pairwise \(q\)-commute and generate an Ore set.
We define
\(L_{\mathbf i}
=A_{\mathbf i}[H_1^{-1},\ldots,H_{r-1}^{-1}]\).
The root grading extends to \(L_{\mathbf i}\) by
\(\deg H_\ell^{-1}=-\widetilde\beta_{2\ell}\).

\begin{lemma}\label{lem:wide-eta-s}
For \(1\le k\le r\), we have
\begin{equation}\label{eq:wide-beta-s}
\widetilde\beta_{2k-1}
=
\gamma_k
-\sum_{p=1}^{k-1}
\bigl(
s_{i_{p+1}}\cdots s_{i_{k-1}}\alpha_{i_k},
\alpha_{\infty_p}
\bigr)\widetilde\beta_{2p},
\end{equation}
where \(\gamma_k=w_{\le k-1}\alpha_{i_k}\) is as in
Definition~\ref{def:PBW-roots}.
More generally, for \(1\le a\le k\), set
\(v_{a,k-1}
=\widetilde w_{\le2a-2}s_{i_a}\cdots s_{i_{k-1}}\).
Then
\begin{equation}\label{eq:vk-1-beta}
\widetilde\beta_{2k-1}
=
v_{a,k-1}\alpha_{i_k}
-\sum_{p=a}^{k-1}
\bigl(
s_{i_{p+1}}\cdots s_{i_{k-1}}\alpha_{i_k},
\alpha_{\infty_p}
\bigr)\widetilde\beta_{2p}.
\end{equation}
Empty products and sums are understood as the identity and zero,
respectively.
\end{lemma}

\begin{proof}
Since \((\alpha_{\infty_p},\alpha_{\infty_p})=2\), the
reflection formula for \(\widetilde s_{\infty_p}\) has no
additional root-length factor. For \(a\le p<k\), it gives
$$
\begin{aligned}
&(v_{p+1,k-1}-v_{p,k-1})\alpha_{i_k}\\
&\quad=
\widetilde w_{\le2p-1}
(\widetilde s_{\infty_p}-1)
s_{i_{p+1}}\cdots s_{i_{k-1}}\alpha_{i_k}\\
&\quad=
-\bigl(
s_{i_{p+1}}\cdots s_{i_{k-1}}\alpha_{i_k},
\alpha_{\infty_p}
\bigr)\widetilde\beta_{2p},
\end{aligned}
$$
since
\(\widetilde w_{\le2p-1}\alpha_{\infty_p}
=\widetilde\beta_{2p}\).
Summing over \(p=a,\ldots,k-1\) and using
\(v_{k,k-1}\alpha_{i_k}=\widetilde\beta_{2k-1}\)
proves \eqref{eq:vk-1-beta}.
Taking \(a=1\), for which \(v_{1,k-1}=w_{\le k-1}\),
gives \eqref{eq:wide-beta-s}.
\end{proof}

Define the separator degree
\(\deg_\infty:\widetilde Q\to\mathbb Z^{r-1}\) by
\(\deg_\infty(\alpha_i)=0\) for \(i\in I\) and
\(\deg_\infty(\alpha_{\infty_\ell})=e_\ell\), where
\(e_1,\ldots,e_{r-1}\) are the standard basis vectors.
Composing this map with the root grading gives a
\(\mathbb Z^{r-1}\)-grading on \(L_{\mathbf i}\).
The \emph{dehomogenized algebra} is the degree-zero subalgebra
of $L_{\mathbf i}$:
\begin{equation}
\D_{\mathbf i}:=(L_{\mathbf i})_0.
\end{equation}

The root \(\widetilde\beta_{2\ell}\) has coefficient \(1\)
at \(\alpha_{\infty_\ell}\) and coefficient \(0\) at
\(\alpha_{\infty_a}\) for \(a>\ell\). Hence
\(\deg_\infty H_\ell
\in e_\ell+\sum_{a<\ell}\mathbb Z e_a\), so these degrees
form a \(\mathbb Z\)-basis of \(\mathbb Z^{r-1}\).
For each \(k\), there are therefore unique integers
\(m_{\ell k}\) satisfying
\begin{equation}\label{eq:normalizing-exponents}
\deg_\infty X_{2k-1}
+\sum_{\ell=1}^{r-1}
m_{\ell k}\deg_\infty H_\ell=0.
\end{equation}
Lemma~\ref{lem:wide-eta-s} gives
\begin{equation}\label{eq:m-ellk}
m_{\ell k}
=
\begin{cases}
\bigl(
s_{i_{\ell+1}}\cdots s_{i_{k-1}}\alpha_{i_k},
\alpha_{\infty_\ell}
\bigr),&\ell<k,\\
0,&\ell\ge k.
\end{cases}
\end{equation}
Define the right-normalized elements
\begin{equation}\label{eq:Y-definition}
Y_k
=
X_{2k-1}H_1^{m_{1k}}\cdots H_{k-1}^{m_{k-1,k}}
\in\D_{\mathbf i}.
\end{equation}
Equation~\eqref{eq:wide-beta-s} also gives
\(\deg Y_k
=\widetilde\beta_{2k-1}
+\sum_{\ell<k}m_{\ell k}\widetilde\beta_{2\ell}
=\gamma_k\), and hence \(\wt Y_k=-\gamma_k\).
These constructions are compatible with prefixes. For
\(\mathbf i_{\le n}=(i_1,\ldots,i_n)\), the enlarged Cartan matrix
is the restriction to \(I\sqcup\{\infty_1,\ldots,\infty_{n-1}\}\).
The natural embedding of quantum groups identifies its PBW generators
and separators with \(X_1,\ldots,X_{2n-1}\) and
\(H_1,\ldots,H_{n-1}\). Formula~\eqref{eq:m-ellk} is unchanged
for \(k\le n\), and all remaining separator exponents vanish. Hence
\[
\D_{\mathbf i_{\le n}}
=\mathbb K\langle Y_1,\ldots,Y_n\rangle
\subseteq\D_{\mathbf i}
\]
as root-graded algebras. This also identifies the
integral PBW spans. In particular, \(Y_k\) depends only on
\((i_1,\ldots,i_k)\).
For the empty word, set
\(\Ahat(1)=\D_{\emptyset}=\mathbb K\).

The matrix below is the leading PBW commutation matrix
for our LS convention.

\begin{lemma}\label{lem:normalized-skew-pairing}
Let \(e_1,\ldots,e_{2r-1}\) be the standard basis of
\(\mathbb Z^{2r-1}\), and set
\(f_k=e_{2k-1}+\sum_{p<k}m_{pk}e_{2p}\) for \(1\le k\le r\).
Let \(\widetilde\Lambda\) be the skew-symmetric matrix with
\(-\widetilde\Lambda_{st}
=(\widetilde\beta_s,\widetilde\beta_t)\) for \(s<t\).
Then
$$
f_k^{\mathsf t}\widetilde\Lambda f_\ell
=-(\gamma_k,\gamma_\ell)
\qquad(1\le k<\ell\le r).
$$
\end{lemma}

\begin{proof}
Let \(T\) and \(\widetilde T\) be the word matrices of
\(\mathbf i\) and \(\mathbf j\), respectively, as in
Lemma~\ref{lem:root-coordinate-matrices}. Put
\[
\begin{gathered}
E=T^{-1},\qquad \widetilde E=\widetilde T^{-1},\\
\Delta_{\mathbf i}=\operatorname{diag}(d_{i_1},\ldots,d_{i_r}),\qquad
\Delta_{\mathbf j}=\operatorname{diag}
(d_{i_1},1,d_{i_2},\ldots,1,d_{i_r}),\\
\mathsf E=\Delta_{\mathbf i}E,\qquad
\widetilde{\mathsf E}=\Delta_{\mathbf j}\widetilde E,\qquad
F=(f_1\ \cdots\ f_r),\qquad
J=(e_1\ e_3\ \cdots\ e_{2r-1}).
\end{gathered}
\]
Then \(J^{\mathsf t}F=I_r\). The same lemma gives
\[
\widetilde\Lambda
=\widetilde{\mathsf E}^{\mathsf t}-\widetilde{\mathsf E},\qquad
\widetilde{\mathsf E}_{st}=(\widetilde\beta_s,\widetilde\beta_t)
\quad(s<t),\qquad
\mathsf E_{pk}=(\gamma_p,\gamma_k)\quad(p<k).
\]
The diagonal entries of \(\widetilde{\mathsf E}\) are
\(\widetilde d_{j_s}\).

We claim that \(\widetilde{\mathsf E}F=J\mathsf E\).
For \(p<k\), equations~\eqref{eq:vk-1-beta}
and~\eqref{eq:m-ellk}, together with
\(\widetilde d_{\infty_p}=1\), give
\[
\begin{aligned}
e_{2p}^{\mathsf t}\widetilde{\mathsf E}f_k
&=m_{pk}
 +\left(\widetilde\beta_{2p},
 \widetilde\beta_{2k-1}
 +\sum_{s=p+1}^{k-1}m_{sk}\widetilde\beta_{2s}\right)\\
&=m_{pk}+(\widetilde\beta_{2p},v_{p+1,k-1}\alpha_{i_k})\\
&=m_{pk}+\left(\alpha_{\infty_p},
 \widetilde s_{\infty_p}s_{i_{p+1}}\cdots
 s_{i_{k-1}}\alpha_{i_k}\right)\\
&=m_{pk}-\left(\alpha_{\infty_p},
 s_{i_{p+1}}\cdots s_{i_{k-1}}\alpha_{i_k}\right)=0.
\end{aligned}
\]
Here
\(\widetilde\beta_{2p}
=\widetilde w_{\le2p-1}\alpha_{\infty_p}\) and
\(v_{p+1,k-1}
=\widetilde w_{\le2p-1}\widetilde s_{\infty_p}
s_{i_{p+1}}\cdots s_{i_{k-1}}\).
For \(p\ge k\), upper triangularity gives zero, since
\(f_k\) is supported in positions at most \(2k-1\).

Similarly, for \(p<k\),
\[
\begin{aligned}
e_{2p-1}^{\mathsf t}\widetilde{\mathsf E}f_k
&=\left(\widetilde\beta_{2p-1},
 \widetilde\beta_{2k-1}
 +\sum_{s=p}^{k-1}m_{sk}\widetilde\beta_{2s}\right)\\
&=(\widetilde\beta_{2p-1},v_{p,k-1}\alpha_{i_k})\\
&=(\alpha_{i_p},s_{i_p}\cdots s_{i_{k-1}}\alpha_{i_k})\\
&=(\gamma_p,\gamma_k)=\mathsf E_{pk}.
\end{aligned}
\]
For \(p=k\), this entry is \(d_{i_k}=\mathsf E_{kk}\);
for \(p>k\), it is \(0=\mathsf E_{pk}\).
This proves the claim. Since \(F^{\mathsf t}J=I_r\), it follows that
\[
\begin{aligned}
F^{\mathsf t}\widetilde{\mathsf E}F&=\mathsf E,\\
F^{\mathsf t}\widetilde\Lambda F
&=F^{\mathsf t}
 (\widetilde{\mathsf E}^{\mathsf t}-\widetilde{\mathsf E})F
=\mathsf E^{\mathsf t}-\mathsf E.
\end{aligned}
\]
Taking the \((k,\ell)\)-entry for \(k<\ell\) gives
\(f_k^{\mathsf t}\widetilde\Lambda f_\ell
=-\mathsf E_{k\ell}=-(\gamma_k,\gamma_\ell)\), as required.
\end{proof}

Write \(H^{\mathbf b}=H_1^{b_1}\cdots H_{r-1}^{b_{r-1}}\),
\(\mathbf m_k=(m_{1k},\ldots,m_{r-1,k})^{\mathsf t}\), and
\(Y^{\mathbf a}=Y_1^{a_1}\cdots Y_r^{a_r}\).

\begin{lemma}\label{lem:PBW-Dbfi}
The monomials \(Y^{\mathbf a}\), with
\(\mathbf a\in\mathbb N^r\), form a
\(\mathbb K\)-basis of \(\D_{\mathbf i}\).
Moreover, \(\deg Y_k=\gamma_k\) for every \(k\in[1,r]\).
\end{lemma}

\begin{proof}
Each \(H_p\) corresponds to a singleton prime chain.
Thus the Ore derivation at its adjoining step is zero,
and every subsequent Ore derivation annihilates \(H_p\),
by Theorems~\ref{thm:cgl-recursive-primes}
and~\ref{thm:cgl-detected-prime}. Consequently,
\(H_pX_s=q^{\widetilde\Lambda_{2p,s}}X_sH_p\)
for all \(p,s\).

It follows that the localized PBW basis of \(L_{\mathbf i}\)
consists of
\(X_1^{a_1}X_3^{a_2}\cdots X_{2r-1}^{a_r}H^{\mathbf b}\),
where \(\mathbf a\in\mathbb N^r\) and
\(\mathbf b\in\mathbb Z^{r-1}\).
These monomials span because separator powers can be moved
to the right. They are linearly independent by clearing a
common right denominator and applying the PBW theorem for
\(A_{\mathbf i}\).

By \eqref{eq:normalizing-exponents}, the separator degree
of such a monomial is
$$
\deg_\infty
\bigl(X_1^{a_1}X_3^{a_2}\cdots X_{2r-1}^{a_r}H^{\mathbf b}\bigr)
=
\sum_{p=1}^{r-1}
\left(b_p-\sum_{k=1}^r a_km_{pk}\right)\deg_\infty H_p.
$$
Since the degrees \(\deg_\infty H_p\) form a
\(\mathbb Z\)-basis, this degree is zero precisely when
\(\mathbf b=\sum_{k=1}^r a_k\mathbf m_k\).
In that case, moving the separator factors shows that the
monomial is a nonzero scalar multiple of \(Y^{\mathbf a}\).
The degree-zero part of the localized PBW basis therefore
gives the required basis of \(\D_{\mathbf i}\).
The identity \(\deg Y_k=\gamma_k\) follows from
Lemma~\ref{lem:wide-eta-s} and \eqref{eq:m-ellk}.
\end{proof}

Let $A_{\mathbf i,\A}=A_{\A}(\mathfrak n(\widetilde w_{\mathbf i}))$
be the standard dual canonical integral form, and define
\[
 L_{\mathbf i,\A}
 :=A_{\mathbf i,\A}[H_1^{-1},\ldots,H_{r-1}^{-1}],
 \qquad
 \D_{\mathbf i,\A}:=(L_{\mathbf i,\A})_0.
\]
For the empty word put $\D_{\emptyset,\A}=\A$.
Since $H_p=q^{-1/2}X_{2p}$, the localized PBW basis is obtained
from the ambient dual PBW basis by allowing arbitrary integer
exponents at the even positions.  All reordering factors are
units of $\A$.  The proof of Lemma~\ref{lem:PBW-Dbfi} therefore
gives
\begin{equation}\label{eq:integral-Y-PBW}
 \D_{\mathbf i,\A}
 =\bigoplus_{\mathbf a\in\mathbb N^r}
   \A Y_1^{a_1}\cdots Y_r^{a_r},
 \qquad
 \mathbb K\otimes_{\A}\D_{\mathbf i,\A}
 =\D_{\mathbf i}.
\end{equation}
We denote by $\D_{\mathbf i,\A}[a,b]$ the $\A$-span of
ordered PBW monomials supported in $[a,b]$. This span is an
algebra: the ambient integral straightening law and commutation
with the separators give integral normalized corrections, and
the degree-zero PBW criterion places them in the same interval.

\begin{proposition}\label{pro:CGL-Dbfi}
The algebra \(\D_{\mathbf i}\) has the symmetric CGL presentation
$$
\D_{\mathbf i}
=
\mathbb K[Y_1]
[Y_2;\sigma_2,\delta_2]\cdots
[Y_r;\sigma_r,\delta_r]
$$
for the action of the root torus \(\mathcal H\) induced
by its \(Q\)-grading.
For \(k<\ell\), its generators satisfy
\begin{equation}\label{eq:common-LS}
Y_kY_\ell-q^{-(\gamma_k,\gamma_\ell)}Y_\ell Y_k
\in
\mathbb K\langle Y_{k+1},\ldots,Y_{\ell-1}\rangle.
\end{equation}
For homogeneous
\(a\in\mathbb K\langle Y_1,\ldots,Y_{\ell-1}\rangle\),
the Ore maps are
\begin{equation}\label{eq:common-sigma}
\sigma_\ell(a)=q^{(\gamma_\ell,\deg a)}a,
\qquad
\delta_\ell(a)=Y_\ell a-\sigma_\ell(a)Y_\ell,
\end{equation}
extended linearly.
\end{proposition}

\begin{proof}
Denote the left-hand side of \eqref{eq:common-LS} by
\(R_{k\ell}\), and set
\(C_{k\ell}
=X_{2k-1}X_{2\ell-1}
-q^{-(\widetilde\beta_{2k-1},\widetilde\beta_{2\ell-1})}
X_{2\ell-1}X_{2k-1}\).
By \eqref{eq:quantum-cell-LS},
\(C_{k\ell}\in
\mathbb K\langle X_{2k},\ldots,X_{2\ell-2}\rangle\).
Moving separator powers to the right and applying
Lemma~\ref{lem:normalized-skew-pairing} gives
\begin{equation}\label{eq:normalized-remainder}
R_{k\ell}
=
\kappa_{k\ell}C_{k\ell}H^{\mathbf m_k+\mathbf m_\ell}
\qquad
\text{for some }\kappa_{k\ell}\in q^{\mathbb Z/2}.
\end{equation}

Expanding \(C_{k\ell}\) in its interval PBW basis shows
that every localized PBW monomial in \(R_{k\ell}\) has the form
\(X_{2k+1}^{a_{k+1}}\cdots
X_{2\ell-3}^{a_{\ell-1}}H^{\mathbf b}\).
Since \(R_{k\ell}\) has separator degree zero, every monomial
with nonzero coefficient has degree zero. The PBW criterion
therefore gives
\(\mathbf b=\sum_{j=k+1}^{\ell-1}a_j\mathbf m_j\).
Thus this monomial is a nonzero scalar multiple of
\(Y_{k+1}^{a_{k+1}}\cdots Y_{\ell-1}^{a_{\ell-1}}\),
proving \eqref{eq:common-LS}.
For \(\ell=k+1\), the intermediate algebra is understood
to be \(\mathbb K\).

It remains to verify the CGL conditions.
The grading makes each \(\sigma_\ell\) an automorphism,
and the definition of \(\delta_\ell\) gives
$$
\begin{aligned}
\delta_\ell(uv)
&=\delta_\ell(u)v+\sigma_\ell(u)\delta_\ell(v),\\
\delta_\ell(Y_k)
&=-q^{(\gamma_k,\gamma_\ell)}R_{k\ell}
\in\mathbb K\langle Y_{k+1},\ldots,Y_{\ell-1}\rangle
\qquad(k<\ell).
\end{aligned}
$$
Hence \(\delta_\ell\) preserves
\(\mathbb K\langle Y_1,\ldots,Y_{\ell-1}\rangle\).
Together with Lemma~\ref{lem:PBW-Dbfi}, these identities
give the asserted iterated Ore presentation.

For local nilpotence, fix $\ell$ and express each
$\delta_\ell(Y_k)$ as a finite linear combination of products
of $Y_{k+1},\ldots,Y_{\ell-1}$. Assign positive integer weights
$d_k$ to the generators, working backwards from $k=\ell-1$.
Choose $d_k$ larger than the weight of every product in the
chosen expression for $\delta_\ell(Y_k)$, where product weights
are additive and scalars have weight zero. This is possible
because only generators with already assigned weights occur.

Let $F_n$ be the span of all products of total weight at
most $n$, and set $F_n=0$ for $n<0$.
The automorphism $\sigma_\ell$ rescales each generator, while
the skew Leibniz rule shows that $\delta_\ell$ lowers weight.
Thus
\[
\sigma_\ell(F_n)=F_n,
\qquad
\delta_\ell(F_n)\subseteq F_{n-1}.
\]
Thus $\delta_\ell^{\,n+1}(F_n)=0$ for every $n\ge0$.
Every element belongs to some $F_n$, so $\delta_\ell$
is locally nilpotent.

Finally, the \(Q\)-grading defines a rational
\(\mathcal H\)-action for which every \(Y_k\) is an
eigenvector. For \(\alpha\in Q\), choose
\(h_\ell(\alpha)=q^{(\gamma_\ell,\alpha)}\) and
\(h_k^*(\alpha)=q^{-(\gamma_k,\alpha)}\).
Then \(\sigma_\ell\) is the restriction of \(h_\ell\),
and, writing
\(\lambda_{\ell k}=q^{(\gamma_\ell,\gamma_k)}\), we have
$$
\begin{aligned}
h_\ell\cdot Y_\ell&=q_{i_\ell}^{2}Y_\ell,
& h_k^*\cdot Y_k&=q_{i_k}^{-2}Y_k,\\
h_k^*\cdot Y_\ell
&=q^{-(\gamma_k,\gamma_\ell)}Y_\ell
=\lambda_{\ell k}^{-1}Y_\ell
&& (k<\ell).
\end{aligned}
$$
Here \((\gamma_k,\gamma_k)=2d_{i_k}\) by Weyl invariance, and
\(q_{i_k}^{\pm2}\) are not roots of unity.
Together with the interval containment for
\(\delta_\ell(Y_k)\), these identities verify the remaining
CGL conditions and symmetry.
\end{proof}

For \(1\le n\le 2r-1\) and \(1\le k\le r\), write
\(A_n=\mathbb K\langle X_1,\ldots,X_n\rangle\) and
\(\D_k=\mathbb K\langle Y_1,\ldots,Y_k\rangle\), with
\(\D_0=\mathbb K\). Put
\(L_k=A_{2k-1}[H_1^{-1},\ldots,H_{k-1}^{-1}]\).
The localized PBW basis gives \((L_k)_0=\D_k\).

Let \(\widetilde y_1,\ldots,\widetilde y_{2r-1}\)
be the distinguished homogeneous-prime sequence of
\(A_{\mathbf i}\), and put \(\mathsf p_k=\widetilde y_{2k-1}\).
Write \(k^-=\max\{a<k\mid i_a=i_k\}\), with
\(\max\emptyset=-\infty\).
By \cite[Theorem~7.1(b)]{GY-integral}, the ambient prime
chains are indexed by the letters of
\(\operatorname{th}(\mathbf i)\). Thus
\(\widetilde y_{2a}=q^{1/2}H_a\), and
\begin{equation}\label{eq:ambient-odd-prime-recursion}
\mathsf p_k=
\begin{cases}
X_{2k-1},&k^-=-\infty,\\
\mathsf p_{k^-}X_{2k-1}-C_k,&k^-\ne-\infty,
\end{cases}
\qquad C_k\in A_{2k-2}.
\end{equation}
Put
$$
\boldsymbol\nu_k
=\sum_{\substack{a\le k\\ i_a=i_k}}\mathbf m_a,
\qquad
Z_k^{\mathrm{raw}}=\mathsf p_kH^{\boldsymbol\nu_k}.
$$
The recursion gives
\(\deg \mathsf p_k=\sum_{a\le k,\ i_a=i_k}\widetilde\beta_{2a-1}\);
hence \eqref{eq:normalizing-exponents} implies
\(\deg_\infty Z_k^{\mathrm{raw}}=0\), so \(Z_k^{\mathrm{raw}}\in\D_k\).

\begin{lemma}\label{lem:predecessor-normalized-algebra}
For every \(k\in[1,r]\), the element \(Z_k^{\mathrm{raw}}\) is a
homogeneous prime element of \(\D_k\) outside
\(\D_{k-1}\).
The predecessor function for the CGL presentation
of \(\D_{\mathbf i}\) is \(p(k)=k^-\); equivalently,
its prime chains are the level sets of \(\eta(k)=i_k\).
\end{lemma}

\begin{proof}
The elements \(\mathsf p_k,H_1,\ldots,H_{k-1}\) are pairwise
nonassociate homogeneous prime elements of \(A_{2k-1}\).
Thus the separator Ore set is disjoint from
\(\mathsf p_kA_{2k-1}\), and \(\mathsf p_k\) remains a normal prime
element of \(L_k\).
Since \(H^{\boldsymbol\nu_k}\) is invertible and
\(Z_k^{\mathrm{raw}}\) has separator degree zero, taking degree-zero
components gives
$$
Z_k^{\mathrm{raw}}L_k=\mathsf p_kL_k,
\qquad
Z_k^{\mathrm{raw}}L_k\cap\D_k=Z_k^{\mathrm{raw}}\D_k=\D_kZ_k^{\mathrm{raw}}.
$$
Consequently,
\(\D_k/Z_k^{\mathrm{raw}}\D_k\hookrightarrow L_k/\mathsf p_kL_k\).
The target is a nonzero domain, so \(Z_k^{\mathrm{raw}}\) is a
normal prime element of \(\D_k\). It is root-homogeneous because
\(\mathsf p_k\) and the separators are root-homogeneous.

Next, \eqref{eq:ambient-odd-prime-recursion} and the
definition of \(\boldsymbol\nu_k\) yield
\begin{equation}\label{eq:normalized-prime-recursion}
Z_k^{\mathrm{raw}}=
\begin{cases}
Y_k,&k^-=-\infty,\\
a_kZ_{k^-}^{\mathrm{raw}}Y_k-d_k,&k^-\ne-\infty,
\end{cases}
\qquad
a_k\in\mathbb K^\times,\quad d_k\in\D_{k-1}.
\end{equation}
Moving separator powers to the right introduces only scalar
factors. The normalized remainder has separator degree zero
and involves only odd generators \(X_{2j-1}\) with \(j<k\).
By the PBW criterion in Lemma~\ref{lem:PBW-Dbfi}, it belongs
to \(\D_{k-1}\).

In particular, \(Z_k^{\mathrm{raw}}\) has degree one in \(Y_k\), so
\(Z_k^{\mathrm{raw}}\notin\D_{k-1}\).
Let \(y_1,\ldots,y_r\) be the distinguished homogeneous-prime
sequence of \(\D_{\mathbf i}\).
By \cite[Theorem~3.2]{GY-integral},
\(Z_k^{\mathrm{raw}}\in\mathbb K^\times y_k\).
Comparing the coefficients of \(Y_k\) in
\eqref{eq:normalized-prime-recursion} and the CGL prime
recursion gives \(p(k)=-\infty\) when \(k^-=-\infty\),
and \(y_{p(k)}\in\mathbb K^\times y_{k^-}\) otherwise.
Since \(y_a\in\D_a\setminus\D_{a-1}\), distinct terms
cannot be scalar multiples.
Hence \(p(k)=k^-\).
\end{proof}

\begin{lemma}\label{lem:terminal-prime-normality}
Fix \(n\in[1,r]\) and put \(w=w_{\le n}\).
For each label \(a\) occurring in \((i_1,\ldots,i_n)\),
let \(g_a\) be the terminal homogeneous prime element
of \(\D_n\) in the chain labelled by \(a\).
Then
\begin{equation}\label{eq:prime-degree}
\deg g_a=(1-w)\varpi_a,
\end{equation}
and, for every root-homogeneous \(u\in\D_n\),
\begin{equation}\label{eq:prime-normality}
g_a u=q^{-((1+w)\varpi_a,\deg u)}u g_a.
\end{equation}
\end{lemma}

\begin{proof}
The prime chains give
\[
\deg g_a
=\sum_{\substack{k\le n\\i_k=a}}\gamma_k
=\sum_{k=1}^n(w_{\le k-1}-w_{\le k})\varpi_a
=(1-w)\varpi_a.
\]
Our Ore convention is
\(\sigma_k(Y_j)=q^{(\gamma_k,\gamma_j)}Y_j\) for \(j<k\).
Thus the normality formula
\cite[Corollary~4.8]{GY-UFD} gives
\(g_aY_j=q^{\chi_{aj}}Y_jg_a\), where
\[
\begin{aligned}
\chi_{aj}
&=-\sum_{\substack{k<j\\i_k=a}}(\gamma_k,\gamma_j)
  +\sum_{\substack{j<k\le n\\i_k=a}}(\gamma_k,\gamma_j)\\
&=-\bigl((1+w-2w_{\le j-1})\varpi_a
       +\delta_{a,i_j}\gamma_j,\gamma_j\bigr)\\
&=-((1+w)\varpi_a,\gamma_j).
\end{aligned}
\]
The last equality uses
\((w_{\le j-1}\varpi_a,\gamma_j)=d_a\delta_{a,i_j}\)
and \((\gamma_j,\gamma_j)=2d_{i_j}\).
Since the \(Y_j\) generate \(\D_n\), this proves
\eqref{eq:prime-normality}.
\end{proof}

\begin{lemma}[Invariant center]\label{lem:center}
For every interval subalgebra \(B\) of \(\D_{\mathbf i}\),
\[
Z(\operatorname{Frac}B)^{\mathcal H}=\mathbb K.
\]
\end{lemma}

\begin{proof}
The assertion is immediate for the empty interval. Otherwise,
let \(n\) be the interval length and \(\mathcal T\) its Cauchon
quantum torus. Then
\(\operatorname{Frac}\mathcal T=\operatorname{Frac}B\),
and the Cauchon generators have the same
\(\mathcal H\)-characters as the corresponding Ore
generators \cite[Section~6.1]{GY-UFD}.

Let \(G\) be the matrix whose columns are the root degrees
of these generators. Apply
Lemma~\ref{lem:root-coordinate-matrices} to the interval
word, and write \(E\) for its upper unitriangular matrix.
Let \(\Delta\) be the diagonal matrix of the symmetrizing
integers at its positions, and put \(\mathsf E=\Delta E\).
The interval degrees differ from those of the interval
word by a common Weyl group action. Consequently,
\[
G^{\mathsf t}D_CCG=\mathsf E+\mathsf E^{\mathsf t}.
\]
With the convention
\(\overline x_j\overline x_k
=q^{\Lambda_{jk}}\overline x_k\overline x_j\),
the Cauchon commutation matrix is
\(\Lambda=\mathsf E^{\mathsf t}-\mathsf E\).
In particular,
\begin{equation}\label{eq:center-joint-kernel}
G\mathbf v=\Lambda\mathbf v=0
\quad\Longrightarrow\quad
2\mathsf E\mathbf v
=(G^{\mathsf t}D_CCG-\Lambda)\mathbf v=0
\quad\Longrightarrow\quad
\mathbf v=0.
\end{equation}
Only the invertibility of \(\mathsf E=\Delta E\) is needed here;
its diagonal entries are positive integers. The Cartan matrix \(C\)
may be singular.

Let $\overline x_1,\ldots,\overline x_n$ be the Cauchon generators.
The Laurent monomials
$\overline x^{\mathbf v}:=\overline x_1^{v_1}\cdots
\overline x_n^{v_n}$, with $\mathbf v\in\mathbb Z^n$,
form a basis of $\mathcal T$ and satisfy
\[
\overline x_j\overline x^{\mathbf v}\overline x_j^{-1}
=q^{(\Lambda\mathbf v)_j}\overline x^{\mathbf v},
\qquad
h\cdot\overline x^{\mathbf v}
=h(G\mathbf v)\overline x^{\mathbf v}
\quad(h\in\mathcal H).
\]
Since $q$ is an indeterminate and the root torus separates
root degrees, these actions fix $\overline x^{\mathbf v}$
precisely when $\Lambda\mathbf v=G\mathbf v=0$.
By linear independence, every monomial occurring in a central
$\mathcal H$-invariant element must satisfy both conditions.
The calculation above forces $\mathbf v=0$. Hence
\begin{equation}\label{eq:center-torus-invariants}
Z(\mathcal T)^{\mathcal H}=\mathbb K.
\end{equation}

To pass to the fraction division ring, take
\(c\in Z(\operatorname{Frac}\mathcal T)^{\mathcal H}\)
and set
\[
J_c=\{a\in\mathcal T\mid ac\in\mathcal T\}.
\]
Clearing a left denominator shows that \(J_c\ne0\).
Since \(c\) is central and \(\mathcal H\)-invariant,
\(J_c\) is a two-sided \(\mathcal H\)-stable ideal.
Choose a nonzero element \(a\in J_c\) with minimal
Laurent support. After multiplication by a Laurent monomial unit and a scalar,
we may assume that its constant coefficient is \(1\).
For every \(h\in\mathcal H\) and every Cauchon generator
\(\overline x_j\), the elements
\[
h\cdot a-a,
\qquad
\overline x_j a\overline x_j^{-1}-a
\]
belong to \(J_c\), and their supports are contained in
\(\operatorname{supp}(a)\setminus\{0\}\).
Minimality forces both elements to vanish.
Thus \(a\in Z(\mathcal T)^{\mathcal H}=\mathbb K\),
and its constant coefficient gives \(a=1\).
It follows that \(c\in\mathcal T\), so
\eqref{eq:center-torus-invariants} implies
\(c\in\mathbb K\).
\end{proof}

\subsection{The comparison isomorphism}

\begin{theorem}\label{thm:comparison}
Let \(C\) be symmetrizable and let
\(\mathbf i=(i_1,\ldots,i_r)\) represent a positive
braid \(b\). There are scalars
\(c_1,\ldots,c_r\in\mathbb K^\times\), chosen
recursively along prefixes, such that
\[
\Phi_{\mathbf i}:\Ahat(b)\xrightarrow{\sim}\D_{\mathbf i},
\qquad P_k^{\mathbf i}\longmapsto c_kY_k,
\]
is an isomorphism of root-graded \(\mathbb K\)-algebras.
\end{theorem}

\begin{proof}
We argue by induction on the word length, starting with the
empty prefix. Suppose that the two
prefix algebras for \((i_1,\ldots,i_n)\) have been
identified with a common root-graded algebra \(B\),
with corresponding generators differing only by nonzero
scalars. Denote these generators
by \(x_1,\ldots,x_n\), and let \(B[a,d]\) be the
subalgebra generated by \(x_a,\ldots,x_d\).
Append \(i=i_{n+1}\), and put
\(w=w_{\le n}\) and \(\gamma=w\alpha_i\).

The CGL presentations give two extensions
\[
B[x;\sigma,\delta],\qquad B[y;\sigma,\delta'],
\qquad
\sigma(a)=q^{(\gamma,\deg a)}a
\]
for homogeneous \(a\in B\).
Both derivations have root degree \(\gamma\), and
\(\delta(x_j),\delta'(x_j)\in B[j+1,n]\).
We prove that
\begin{equation}\label{eq:derivation-proportionality}
\delta=c\delta'
\qquad\text{for some }c\in\mathbb K^\times.
\end{equation}

If \(i\) does not occur in the prefix, all root degrees
of \(B\) have zero \(\alpha_i\)-coefficient, whereas
\(\gamma\) has coefficient one. Homogeneity therefore
forces \(\delta=\delta'=0\), and we may take \(c=1\).

Otherwise, let
\(p=\max\{j\le n\mid i_j=i\}\) and put
\(M=B[p+1,n]\). Let \(\ell_i(\beta)\) be the
\(\alpha_i\)-coefficient of \(w_{\le p}^{-1}\beta\).
Since none of \(i_{p+1},\ldots,i_n\) equals \(i\),
\[
\ell_i(\gamma_p)=-1,
\qquad \ell_i(\gamma_j)=0\quad(p<j\le n),
\qquad \ell_i(\gamma)=1.
\]
Comparison of degrees in the interval straightening
relations gives
\begin{equation}\label{eq:vanishing-after-p}
\delta(x_j)=\delta'(x_j)=0\quad(j>p),
\qquad
x_pa=q^{-(\gamma_p,\deg a)}ax_p
\quad(a\in M\text{ homogeneous}).
\end{equation}
Indeed, the straightening correction between \(x_p\)
and a generator \(x_j\), with \(j>p\), would have
\(\ell_i\)-degree \(-1\), whereas its intermediate
algebra has \(\ell_i\)-degree zero.

Set \(d=\delta(x_p)\) and \(d'=\delta'(x_p)\).
These elements belong to \(M\) and have degree
\(\gamma_p+\gamma\).
Lemma~\ref{lem:nonvanishing}, applied after removing
the common braid prefix, gives \(d\ne0\).
For the second extension,
Lemmas~\ref{lem:predecessor-normalized-algebra}
and~\ref{lem:cgl-interval-chains} imply that
\(\delta'|_{B[p,n]}\ne0\). Together with
\eqref{eq:vanishing-after-p}, this gives \(d'\ne0\).

For homogeneous \(a\in M\), applying \(\delta\) to
the commutation relation in
\eqref{eq:vanishing-after-p} gives
\[
da
=q^{-(\gamma_p,\deg a)}\sigma(a)d
=q^{(\gamma-\gamma_p,\deg a)}ad.
\]
The same identity holds for \(d'\).
Thus \((d')^{-1}d\) commutes with \(M\) and lies in the center
of \(\operatorname{Frac}M\). It is \(\mathcal H\)-invariant
because \(d\) and \(d'\) have the same root degree.
Lemma~\ref{lem:center} gives
\((d')^{-1}d=c\in\mathbb K^\times\).

Set \(\varepsilon=\delta-c\delta'\). Then
\begin{equation}\label{eq:epsilon-vanishing}
\varepsilon(x_j)=0\qquad(p\le j\le n).
\end{equation}
We claim that \(\varepsilon=0\).
The map \(\varepsilon\) is a homogeneous
\(\sigma\)-derivation of degree \(\gamma\), and
\(\varepsilon(x_j)\in B[j+1,n]\).
The filtration argument in Proposition~\ref{pro:CGL-Dbfi}
shows that \(\varepsilon\) is locally nilpotent.
Give a new Ore generator \(z\) root degree \(\gamma\).
Since \(\varepsilon\) is homogeneous, the root-torus action
extends to \(B[z;\sigma,\varepsilon]\).
The torus elements
\[
h_{n+1}(\beta)=q^{(\gamma,\beta)},
\qquad h_{n+1}^*(\beta)=q^{-(\gamma,\beta)}
\]
have eigenvalues \(q_i^2\) and \(q_i^{-2}\) on \(z\).
The earlier elements \(h_j,h_j^*\) still have the required
actions; in particular,
\(h_j^*\cdot z=q^{-(\gamma_j,\gamma)}z\).
Thus \(B[z;\sigma,\varepsilon]\) is again a symmetric
CGL extension for the same root torus.

Suppose that \(\varepsilon\ne0\), and let \(g_i\)
be the terminal homogeneous prime of \(B\) indexed
by \(p\). We claim that \(\varepsilon(g_i)\ne0\).
Assume otherwise and set
\[
\mu=q^{(\gamma,(1-w)\varpi_i)},
\qquad
\chi_i(\beta)=q^{-((1+w)\varpi_i,\beta)}.
\]
By \eqref{eq:prime-degree} and
\eqref{eq:prime-normality},
\(\sigma(g_i)=\mu g_i\) and
\(g_i a=\chi_i(\beta)ag_i\) for homogeneous
\(a\in B\) of degree \(\beta\).
Apply \(\varepsilon\) to this last identity.
Since \(\varepsilon(g_i)=0\) and
\(\deg\varepsilon(a)=\beta+\gamma\), we obtain
\[
\mu\chi_i(\beta+\gamma)\varepsilon(a)g_i
=\chi_i(\beta)\varepsilon(a)g_i.
\]
On the other hand,
\[
\mu\chi_i(\gamma)
=q^{(\gamma,(1-w)\varpi_i)
      -(\gamma,(1+w)\varpi_i)}
=q^{-2(\gamma,w\varpi_i)}
=q^{-2d_i}=q_i^{-2}\ne1.
\]
Since \(B\) is a domain and \(g_i\ne0\), we obtain
\(\varepsilon(a)=0\) for every homogeneous \(a\in B\),
contrary to our assumption.
Hence \(\varepsilon(g_i)\ne0\).

Theorem~\ref{thm:cgl-detected-prime} now identifies
\(p\) as the predecessor of \(z\) in this new CGL
extension. Applying
Lemma~\ref{lem:cgl-interval-chains} to its interval
subalgebra generated by \(x_p,\ldots,x_n,z\)
gives \(\varepsilon|_{B[p,n]}\ne0\).
But \eqref{eq:epsilon-vanishing} and the skew Leibniz
rule imply \(\varepsilon|_{B[p,n]}=0\).
This contradiction proves \eqref{eq:derivation-proportionality}.

The prefix isomorphism therefore extends by
\(x\mapsto cy\), since
\[
(cy)a-\sigma(a)(cy)=c\delta'(a)=\delta(a).
\]
This extension is an isomorphism by the PBW bases.
Induction proves the theorem.
\end{proof}

\begin{lemma}\label{lem:ambient-primitive-LS}
Let $p<k$ satisfy $p=k^-$, and put
\[
 C'_{p,k}:=Y_pY_k-q^{-(\gamma_p,\gamma_k)}Y_kY_p.
\]
Put $i=i_p=i_k$ and $\zeta_i=1-q_i^2$.
Then $C'_{p,k}/\zeta_i$ belongs to the integral interval algebra
$\D_{\mathbf i,\A}[p+1,k-1]$ and is primitive there.
\end{lemma}
\begin{proof}
For the reduced word $\operatorname{th}(\mathbf i)$, the
coordinate-ring realization of the zeroth bosonic layer identifies
the dual root vectors $X_s$ with the bosonic cuspidal generators.
Writing
$\mathbf j=\operatorname{th}(\mathbf i)$, one has
\[
\psi_0(X_s)
=\mathbf T_{j_1}\cdots\mathbf T_{j_{s-1}}
  (q_{j_s}^{1/2}f_{j_s,0}).
\]
This identity follows from \cite[(1.20), (2.12)--(2.13),
and (3.10)]{KKOP-braid}, with $\psi_0$ denoting the
coordinate-ring map fixed in Section~\ref{sec:bosonic}.
The dual PBW bases show that this identification is integral.
Lemma~\ref{lem:primitive-LS}, applied to the enlarged
symmetrizable Cartan datum, shows that the ambient straightening
correction between $X_{2p-1}$ and $X_{2k-1}$, divided by
$\zeta_i=1-q_i^2$, is integral
and primitive in the ambient dual PBW lattice.

Moving the separator factors to the right writes $C'_{p,k}$
as this ambient correction times a separator Laurent monomial
and an $\A$-unit. In the localized PBW basis, this multiplication
shifts the even exponents and rescales coefficients by powers
of $q$, so it preserves primitivity. The degree-zero localized
PBW basis differs from the $Y$-PBW basis only by $\A$-units.
Thus $C'_{p,k}/\zeta_i$ is integral and primitive in
$\D_{\mathbf i,\A}$. The interval straightening relation and
PBW independence place it in $\D_{\mathbf i,\A}[p+1,k-1]$,
where it remains primitive.
\end{proof}

\begin{theorem}[Integral comparison]\label{thm:integral-comparison}
The comparison isomorphism can be chosen with
$c_k\in \A^\times$ for all $k$. It then restricts to an
isomorphism
\[
 \Phi_{\mathbf i,\A}:\widehat A_{\A}(b)
       \xrightarrow{\sim}\D_{\mathbf i,\A},
 \qquad P_k^{\mathbf i}\longmapsto c_kY_k.
\]
\end{theorem}
\begin{proof}
We induct on the word length, choosing $c_k=1$ whenever
$i_k$ has no predecessor. Assume that the prefix comparison
is integral. In the comparison over $\mathbb K$, the two
derivations satisfy $\delta=c\delta'$, with $c$ determined by
$d=cd'$. Here $d,d'$ are their nonzero values on the last
preceding occurrence of the appended letter $i$. Put $\zeta_i=1-q_i^2$.
The ordinary Ore correction differs from the corresponding
forward LS correction only by a signed power of $q$.
The prefix isomorphism identifies the two integral PBW bases
up to the chosen $\A$-units, and hence identifies their
interval lattices supported after the last preceding occurrence.
Lemmas~\ref{lem:primitive-LS} and
\ref{lem:ambient-primitive-LS} show that $d/\zeta_i$ and
$d'/\zeta_i$ are primitive in this common interval lattice.
By Lemma~\ref{lem:integral-interval-saturation} and the
$Y$-PBW decomposition \eqref{eq:integral-Y-PBW}, that lattice is
a direct summand of the common free integral prefix algebra.
Its PBW basis extends to the prefix PBW basis, so both elements
remain primitive there. Applying
Lemma~\ref{lem:primitive-rescaling} to
$d/\zeta_i=c\,d'/\zeta_i$ gives $c\in\A^\times$.
The resulting diagonal map identifies the integral PBW bases
up to units, proving the assertion.
\end{proof}

\section{Quantum cluster structures and categorification}
\label{sec:cluster-category}

\subsection{Integral quantum cluster structures}
\label{sec:cluster-dehomogenization}

Set \(\mathbb A=\mathbb Z[q^{\pm1/2}]\) and
\(\mathbb K=\mathbb Q(q^{1/2})\).
Throughout this subsection, \(C\) is symmetrizable, with
\(d_ic_{ij}=d_jc_{ji}\), \(d_i\in\mathbb Z_{>0}\), and
\(q_i=q^{d_i}\). The separators have symmetrizer entries
\(d_{\infty_p}=1\), as in the thickening construction above.
For the empty word, set \(\D_{\emptyset,\mathbb A}=\mathbb A\);
the assertions below are immediate in this case. Henceforth assume
\(r\ge1\). We use the integral forms from
\eqref{eq:integral-Y-PBW} and the multiplication convention
of Section~\ref{sec:CGL}.
Recall that \(H_p=q^{-1/2}X_{2p}\) is the normalized frozen
variable corresponding to the separator \(\infty_p\).
In particular, the monomials in the \(Y_k\) form an
\(\mathbb A\)-basis of \(\D_{\mathbf i,\mathbb A}\), and
\(\mathbb K\otimes_{\mathbb A}\D_{\mathbf i,\mathbb A}
\cong\D_{\mathbf i}\).

Write \(\mathbf j=\operatorname{th}(\mathbf i)\).
With our multiplication and seed conventions,
\cite[Theorem~7.3]{GY-integral} gives
\begin{equation}\label{eq:ambient-integral-cluster}
A_{\mathbf i,\mathbb A}
=\mathcal A_{\mathbb A}(\mathbf s(\mathbf j)),
\end{equation}
where none of the frozen variables is inverted.
Thus \(L_{\mathbf i,\mathbb A}\) is obtained by inverting
only the separator frozen variables.

Denote the initial variables of \(\mathbf s(\mathbf j)\)
by \(\widetilde Z_1,\ldots,\widetilde Z_{2r-1}\), so that
\(\widetilde Z_{2p}=H_p\).
Let \(V=(v_{ak})\in M_r(\mathbb Z)\), where
\(v_{ak}=1\) if \(a\le k\) and \(i_a=i_k\), and zero
otherwise. Write \(\mathbf v_k\) for its \(k\)-th column.
The leading PBW exponent of \(\widetilde Z_{2k-1}\) is
\(\sum_a v_{ak}e_{2a-1}\).
Set
\[
M=(m_{pk})_{1\le p<r,\,1\le k\le r},
\qquad N=MV=(\mathbf n_1\ \cdots\ \mathbf n_r).
\]
Then \(\widetilde Z_{2k-1}H^{\mathbf n_k}\) has separator
degree zero and leading \(Y\)-PBW exponent \(\mathbf v_k\).

Index the ambient seed pattern by a mutation tree
\(\mathbb T\) with initial vertex \(t_0\), and identify
ambient mutable position \(2k-1\) with
\(k\in\operatorname{ex}(\mathbf i)\).
The initial seed is root-homogeneous by
\cite[Proposition~7.2]{GY-integral}. Since mutation preserves
homogeneity, every ambient cluster variable is root-homogeneous.
Write \(\widetilde Z_{t,s}\) for the variables at \(t\).
The separator variables remain \(\widetilde Z_{t,2p}=H_p\).

Order the ambient toric-frame coordinates as
\[
\widetilde{\mathbf Z}_t
=(\widetilde Z_{t,1},\widetilde Z_{t,3},\ldots,
  \widetilde Z_{t,2r-1},H_1,\ldots,H_{r-1}),
\]
and write \(\widetilde\Omega_t\) for their commutation
matrix. The associated toric frame is
\[
\widetilde{\mathsf M}_t(\mathbf u)
=q^{-\frac12\sum_{a<b}u_au_b(\widetilde\Omega_t)_{ab}}
\widetilde{\mathbf Z}_t^{\mathbf u}.
\]
It satisfies
\(\widetilde{\mathsf M}_t(\mathbf u)
\widetilde{\mathsf M}_t(\mathbf v)
=q^{\mathbf u^{\mathsf t}\widetilde\Omega_t\mathbf v/2}
\widetilde{\mathsf M}_t(\mathbf u+\mathbf v)\).

Since the separator degrees form a \(\mathbb Z\)-basis,
there is a unique \(\mathbf n_{t,k}\in\mathbb Z^{r-1}\)
such that
\[
\deg_\infty\widetilde Z_{t,2k-1}
+\sum_p n_{t,pk}\deg_\infty H_p=0.
\]
Set
\begin{equation}\label{eq:normalized-toric-frame}
N_t=(\mathbf n_{t,1}\ \cdots\ \mathbf n_{t,r}),
\qquad K_t=\begin{pmatrix}I_r\\N_t\end{pmatrix},
\qquad
\mathsf M_t(\mathbf a)=\widetilde{\mathsf M}_t(K_t\mathbf a),
\qquad Z_{t,k}=\mathsf M_t(e_k).
\end{equation}
Here \(N_{t_0}=N\), and
\(Z_{t,k}\in\D_{\mathbf i,\mathbb A}\).
Write \(Z_k=Z_{t_0,k}\). These variables are related to the primes constructed above by
\begin{equation}\label{eq:raw-seed-normalization}
Z_k=u_k Z_k^{\mathrm{raw}},
\qquad u_k\in q^{\mathbb Z/2}\subset\mathbb A^\times.
\end{equation}
Indeed, \(\mathbf n_k=\boldsymbol\nu_k\), and the ambient normalized
minor \(\widetilde Z_{2k-1}\) is a scalar multiple of
\(\mathsf p_k=\widetilde y_{2k-1}\).
The recursive prime has leading coefficient one in the ascending
PBW basis, and the normalized minor has leading coefficient in
\(q^{\mathbb Z/2}\). Comparing these coefficients and applying the
toric-frame multiplication rule gives
\eqref{eq:raw-seed-normalization}. We retain the superscript
\(\mathrm{raw}\) when referring to the prime recursion.

\begin{lemma}\label{lem:ZHb}
For \(\mathbf a\in\mathbb N^r\) and
\(\mathbf b\in\mathbb Z^{r-1}\),
\[
\widetilde Z_{t,1}^{a_1}\widetilde Z_{t,3}^{a_2}
\cdots\widetilde Z_{t,2r-1}^{a_r}H^{\mathbf b}
\in\D_{\mathbf i,\mathbb A}
\quad\Longleftrightarrow\quad
\mathbf b=N_t\mathbf a.
\]
In this case the displayed element is
\(q^{d/2}Z_{t,1}^{a_1}\cdots Z_{t,r}^{a_r}\)
for some \(d\in\mathbb Z\).
\end{lemma}

\begin{proof}
Its separator degree is
\(\sum_p(b_p-(N_t\mathbf a)_p)\deg_\infty H_p\).
Since the separator degrees form a basis, this vanishes exactly
when the stated equality holds. The toric-frame multiplication
rule then gives the second assertion, with
\(q^{d/2}\) a unit of \(\mathbb A\).
\end{proof}

\begin{lemma}\label{lem:normalized-seed-commutation}
The initial normalized variables satisfy
\[
Z_kZ_\ell=q^{(\Lambda_{\mathbf i})_{k\ell}}Z_\ell Z_k,
\qquad
\Lambda_{\mathbf i}
=V^{\mathsf t}(E^{\mathsf t}\Delta_{\mathbf i}
              -\Delta_{\mathbf i}E)V,
\]
where \(E=T^{-1}\) is the word matrix from
Lemma~\ref{lem:root-coordinate-matrices} and
\(\Delta_{\mathbf i}=\operatorname{diag}(d_{i_1},\ldots,d_{i_r})\).
For \(k<\ell\), its entries are
\[
(\Lambda_{\mathbf i})_{k\ell}
=\bigl((1-w_{\le k})\varpi_{i_k},
       (1+w_{\le\ell})\varpi_{i_\ell}\bigr).
\]
\end{lemma}

\begin{proof}
Put \(\Lambda_0=E^{\mathsf t}\Delta_{\mathbf i}
-\Delta_{\mathbf i}E\).
For \(s<t\), we have
\((\Lambda_0)_{st}=-(\gamma_s,\gamma_t)\), so
\eqref{eq:common-LS} reads
\[
Y_sY_t-q^{(\Lambda_0)_{st}}Y_tY_s
\in\mathbb K\langle Y_{s+1},\ldots,Y_{t-1}\rangle.
\]
Write \(Y^{\mathbf a}=Y_1^{a_1}\cdots Y_r^{a_r}\),
and use the right lexicographic order on PBW exponents.
Define
\(\theta(\mathbf a,\mathbf b)
:=-\sum_{s<t}a_tb_s(\Lambda_0)_{st}.
\)
In straightening \(Y^{\mathbf a}Y^{\mathbf b}\), moving
\(Y_t\) from the first factor past \(Y_s\) from the second,
with \(s<t\), contributes \(q^{-(\Lambda_0)_{st}}\) to the
leading coefficient.
Thus Lemma~\ref{lem:leading-commutation} gives
\[
Y^{\mathbf a}Y^{\mathbf b}
=
q^{\theta(\mathbf a,\mathbf b)}Y^{\mathbf a+\mathbf b}
+\text{lower PBW terms}.
\]
Moreover,
\(
\theta(\mathbf a,\mathbf b)-\theta(\mathbf b,\mathbf a)
=\mathbf a^{\mathsf t}\Lambda_0\mathbf b.
\)
The normalized prime recursion and
\eqref{eq:raw-seed-normalization} give the leading PBW terms
of the \(Z_k\):
\[
Z_k
=\kappa_kY^{\mathbf v_k}
+\sum_{\mathbf a\prec\mathbf v_k}
  \kappa_{k,\mathbf a}Y^{\mathbf a},
\qquad
\kappa_k\in\mathbb K^\times,
\qquad
\mathbf v_k
=\sum_{\substack{p\le k\\i_p=i_k}}\mathbf e_p.
\]
Indeed, when \(k^-=-\infty\), the element \(Z_k\) is a
nonzero scalar multiple of \(Y_k\).
Otherwise, the recursion expresses \(Z_k\) as a nonzero
scalar multiple of \(Z_{k^-}Y_k\), plus an element of
\(\mathbb K\langle Y_1,\ldots,Y_{k-1}\rangle\).
Every PBW term of the latter has zero \(k\)-th coordinate,
whereas the leading exponent of the product is
\(\mathbf v_{k^-}+\mathbf e_k=\mathbf v_k\).
The assertion therefore follows by induction.

By construction, the \(Z_k\) are Laurent monomials in
one ambient quantum seed. Hence they pairwise
\(q\)-commute: for some \(\xi_{k\ell}\in\mathbb K^\times\),
\[
Z_kZ_\ell=\xi_{k\ell}Z_\ell Z_k.
\]
The right lexicographic order is compatible with addition,
so products involving lower PBW terms cannot contribute
to the coefficient of
\(Y^{\mathbf v_k+\mathbf v_\ell}\).
Comparing this coefficient on both sides gives
\[
\kappa_k\kappa_\ell
q^{\theta(\mathbf v_k,\mathbf v_\ell)}
=
\xi_{k\ell}\kappa_k\kappa_\ell
q^{\theta(\mathbf v_\ell,\mathbf v_k)}.
\]
Consequently,
\[
\xi_{k\ell}
=q^{\theta(\mathbf v_k,\mathbf v_\ell)
     -\theta(\mathbf v_\ell,\mathbf v_k)}
=q^{\mathbf v_k^{\mathsf t}\Lambda_0\mathbf v_\ell}.
\]
Since \(V=(\mathbf v_1\ \cdots\ \mathbf v_r)\), the
commutation relation holds with
\(\Lambda_{\mathbf i}=V^{\mathsf t}\Lambda_0V\).

For a simple-root index \(a\), put
\[
\zeta_{t,a}
:=\sum_{\substack{u\le t\\i_u=a}}\gamma_u
=(1-w_{\le t})\varpi_a.
\]
The second equality follows by summing
\[
w_{\le u-1}\varpi_a-w_{\le u}\varpi_a
=\delta_{i_u,a}\gamma_u.
\]
Fix \(\ell\) and set \(a=i_\ell\).
For \(p\le\ell\), direct summation yields
\[
\begin{aligned}
(\Lambda_0\mathbf v_\ell)_p
&=\sum_{\substack{u<p\\i_u=a}}(\gamma_p,\gamma_u)
 -\sum_{\substack{p<u\le\ell\\i_u=a}}(\gamma_p,\gamma_u)\\
&=-(\gamma_p,\zeta_{\ell,a}
             -2\zeta_{p-1,a}
             -\delta_{i_p,a}\gamma_p)\\
&=(\gamma_p,(1+w_{\le\ell})\varpi_a)
 -2(\gamma_p,w_{\le p-1}\varpi_a)
 +\delta_{i_p,a}(\gamma_p,\gamma_p)\\
&=(\gamma_p,(1+w_{\le\ell})\varpi_a).
\end{aligned}
\]
The last equality follows from Weyl invariance:
\[
(\gamma_p,w_{\le p-1}\varpi_a)
=(\alpha_{i_p},\varpi_a)=d_{i_p}\delta_{i_p,a},
\qquad
(\gamma_p,\gamma_p)=2d_{i_p}.
\]
Therefore, for \(k<\ell\),
\[
\begin{aligned}
(\Lambda_{\mathbf i})_{k\ell}
&=\sum_{\substack{p\le k\\i_p=i_k}}
  (\Lambda_0\mathbf v_\ell)_p\\
&=\bigl(\zeta_{k,i_k},(1+w_{\le\ell})\varpi_{i_\ell}\bigr)\\
&=\bigl((1-w_{\le k})\varpi_{i_k},
        (1+w_{\le\ell})\varpi_{i_\ell}\bigr).
\end{aligned}
\]
\end{proof}

Write the reordered ambient exchange matrix as
\(\widetilde B_t=\binom{B_t}{C_t}\), with odd rows in
\(B_t\) and separator rows in \(C_t\).
Its columns are indexed by \(\operatorname{ex}(\mathbf i)\).

\begin{proposition}\label{pro:normalized-mutation}
Dehomogenization commutes with quantum mutation. More explicitly,  
if \(t_0\xrightarrow{2k_1-1}\cdots\xrightarrow{2k_m-1}t\) is a mutation
sequence in \(\mathcal A_q(\mathbf s(\operatorname{th}(\mathbf i)))\), then
\[
(\mathsf M_t,B_t)
=\mu_{k_m}\cdots\mu_{k_1}
 (\mathsf M_{t_0},B_{\mathbf i}).
\]
All normalizations and exchange relations are defined
over \(\mathbb A\).
\end{proposition}

\begin{proof}
The word formula in Section~\ref{app:word-seeds} gives
\(B_{t_0}=B_{\mathbf i}\). Deleting the frozen separator
rows commutes with matrix mutation, so \(B_t\) is obtained
from \(B_{\mathbf i}\) by the same mutation sequence.

Put \(h_p=\deg_\infty H_p\) and
\(g_{t,s}=\deg_\infty\widetilde Z_{t,2s-1}
=-\sum_p n_{t,ps}h_p\).
Homogeneity of the ambient exchange relation in direction
\(k\) gives
\[
0=\sum_s(B_t)_{sk}g_{t,s}+\sum_p(C_t)_{pk}h_p
=\sum_p\bigl((C_t)_{pk}-(N_tB_t)_{pk}\bigr)h_p.
\]
Hence \(C_t=N_tB_t\) and \(\widetilde B_t=K_tB_t\).
Set \(\Omega_t=K_t^{\mathsf t}\widetilde\Omega_tK_t\).
The compatibility of the ambient seed gives
\begin{equation}\label{eq:normalized-compatibility}
B_t^{\mathsf t}\Omega_t
=\widetilde B_t^{\mathsf t}\widetilde\Omega_tK_t
=2(d_{i_k}\delta_{kj})_{k\in\operatorname{ex}(\mathbf i),\,j\in[1,r]}.
\end{equation}
The diagonal coefficients \(2d_{i_k}\) are unchanged by mutation.
Lemma~\ref{lem:normalized-seed-commutation} identifies
\(\Omega_{t_0}=\Lambda_{\mathbf i}\).
Since \(K_{t_0}\) is injective, the corresponding toric
monomials are linearly independent. The triangular prime recursion
\eqref{eq:normalized-prime-recursion} shows that
\(Z_1^{\mathrm{raw}},\ldots,Z_r^{\mathrm{raw}}\) generate
\(\operatorname{Frac}\D_{\mathbf i}\). By
\eqref{eq:raw-seed-normalization}, the same is true of
\(Z_1,\ldots,Z_r\).
Thus \((\mathsf M_{t_0},B_{\mathbf i})\) is a quantum seed.

Fix \(t\xrightarrow{k}t'\). Write \(\mathbf b\) for the
\(k\)-th column of \(B_t\) and \(\mathbf c=N_t\mathbf b\).
Taking separator degrees in the positive exchange monomial gives
\begin{equation}\label{eq:normalization-mutation}
\mathbf n_{t',k}
=-\mathbf n_{t,k}+N_t[\mathbf b]_+-[\mathbf c]_+,
\qquad
\mathbf n_{t',j}=\mathbf n_{t,j}\quad(j\ne k).
\end{equation}
Since \(N_t\mathbf b=\mathbf c\), the first equality also
holds with the positive parts replaced by
\([-\mathbf b]_+\) and \([-\mathbf c]_+\).
Set
\[
\epsilon_k=(e_k,0),\quad
u_\pm=(-e_k+[\pm\mathbf b]_+,[\pm\mathbf c]_+),
\quad h=(0,\mathbf n_{t',k}).
\]
Then
\begin{equation}\label{eq:exchange-lattice-normalization}
u_\pm+h=K_t(-e_k+[\pm\mathbf b]_+),
\qquad
\widetilde Z_{t',2k-1}
=\widetilde{\mathsf M}_t(u_+)
 +\widetilde{\mathsf M}_t(u_-).
\end{equation}
Compatibility and the separator support of \(h\) imply
\[
(u_+-u_-)^{\mathsf t}\widetilde\Omega_t h
=2d_{i_k}\epsilon_k^{\mathsf t}h=0.
\]
Denote the common value
\(u_+^{\mathsf t}\widetilde\Omega_t h
=u_-^{\mathsf t}\widetilde\Omega_t h\) by \(\theta\).
The two exchange summands, and therefore their sum,
\(q\)-commute with \(\widetilde{\mathsf M}_t(h)\) with
exponent \(\theta\).

Let \(\mathsf E_k\) be the identity with \(k\)-th column
\(-e_k+[\mathbf b]_+\), and put
\[
\widetilde{\mathsf E}_k
=\begin{pmatrix}
\mathsf E_k&0\\
[\mathbf c]_+e_k^{\mathsf t}&I_{r-1}
\end{pmatrix}.
\]
Equation~\eqref{eq:normalization-mutation} gives
\(\widetilde{\mathsf E}_kK_{t'}=K_t\mathsf E_k\).
The ambient commutation-matrix mutation therefore gives
\[
\Omega_{t'}
=K_{t'}^{\mathsf t}\widetilde{\mathsf E}_k^{\mathsf t}
 \widetilde\Omega_t\widetilde{\mathsf E}_kK_{t'}
=\mathsf E_k^{\mathsf t}\Omega_t\mathsf E_k.
\]

Since \(h\) is supported on the unchanged separator
variables, \(\widetilde{\mathsf M}_{t'}(h)
=\widetilde{\mathsf M}_t(h)\). The toric multiplication
rule and \(K_{t'}e_k=\epsilon_k+h\) therefore give
\[
\widetilde Z_{t',2k-1}\widetilde{\mathsf M}_t(h)
=q^{\theta/2}Z_{t',k}.
\]
Using \eqref{eq:exchange-lattice-normalization}, we obtain
\[
\begin{aligned}
Z_{t',k}
&=q^{-\theta/2}
  \bigl(\widetilde{\mathsf M}_t(u_+)
       +\widetilde{\mathsf M}_t(u_-)\bigr)
  \widetilde{\mathsf M}_t(h)\\
&=\widetilde{\mathsf M}_t(u_++h)
  +\widetilde{\mathsf M}_t(u_-+h)\\
&=\mathsf M_t(-e_k+[\mathbf b]_+)
  +\mathsf M_t(-e_k+[-\mathbf b]_+).
\end{aligned}
\]
This is the quantum exchange relation for the normalized
variables. The other variables are unchanged, so induction proves
the assertion. All scalar factors are powers of \(q^{1/2}\),
and the calculation is therefore valid over \(\mathbb A\).
\end{proof}

\begin{theorem}\label{thm:cluster-dehomogenization}
For every symmetrizable generalized Cartan matrix and every
positive word \(\mathbf i\), the assignment
\(M_{\mathbf i}(e_k)\mapsto Z_k\) induces an isomorphism
of \(\mathbb A\)-algebras
\[
\Psi_{\mathbf i,\mathbb A}:
\mathcal A_{\mathbb A}(\mathbf s(\mathbf i))
\xrightarrow{\sim}\D_{\mathbf i,\mathbb A}.
\]
No frozen variable of \(\mathbf s(\mathbf i)\) is inverted.
After extension of scalars to \(\mathbb K\), this gives
\(\mathcal A_q(\mathbf s(\mathbf i))\cong\D_{\mathbf i}\).
\end{theorem}

\begin{proof}
The initial toric-frame embedding and
Proposition~\ref{pro:normalized-mutation} give an injective
$\mathbb A$-algebra homomorphism $\Psi_{\mathbf i,\mathbb A}$
whose image is generated by all $Z_{t,k}$.

By \eqref{eq:ambient-integral-cluster},
$L_{\mathbf i,\mathbb A}$ is generated by the ambient
cluster variables $\widetilde Z_{t,2k-1}$ and the
separator units $H_p^{\pm1}$. Since
\[
\widetilde Z_{t,2k-1}
=q^{a_{t,k}/2}Z_{t,k}H^{-\mathbf n_{t,k}},
\qquad a_{t,k}\in\mathbb Z,
\]
it is equivalently generated by all $Z_{t,k}$ and
$H_p^{\pm1}$.

The separators $q$-commute with every normalized cluster
variable. Moving them to the right therefore expresses
every element of $L_{\mathbf i,\mathbb A}$ as an
$\mathbb A$-linear combination of monomials
\[
Z_{t_1,k_1}\cdots Z_{t_s,k_s}H^{\mathbf b},
\qquad \mathbf b\in\mathbb Z^{r-1}.
\]
Each $Z_{t,k}$ has degree zero, while the degrees of the
$H_p$ form a $\mathbb Z$-basis. Taking the degree-zero
component thus retains exactly the terms with
$\mathbf b=0$. Hence
$\D_{\mathbf i,\mathbb A}=(L_{\mathbf i,\mathbb A})_0$
is precisely the $\mathbb A$-algebra generated by all
$Z_{t,k}$, proving surjectivity.

No frozen variable of $\mathbf s(\mathbf i)$ is inverted.
Extension of scalars to $\mathbb K$ gives the final assertion.
\end{proof}

\begin{corollary}\label{cor:isoAb}
For every symmetrizable generalized Cartan matrix and every
positive word \(\mathbf i\) representing \(b\), there is
an isomorphism of \(\mathbb A\)-algebras
\[
\widehat A_{\mathbb A}(b)
\cong\mathcal A_{\mathbb A}(\mathbf s(\mathbf i)).
\]
After extension of scalars, this gives
\(\widehat A(b)\cong\mathcal A_q(\mathbf s(\mathbf i))\).
\end{corollary}

\begin{proof}
Combine Theorems~\ref{thm:integral-comparison}
and~\ref{thm:cluster-dehomogenization}.
\end{proof}

\begin{remark}\label{rem:scope-cluster-structure}
For each choice of $\mathbf i$ and comparison map
$\Phi_{\mathbf i,\A}$, the corollary gives a quantum cluster
structure on $\widehat A_{\A}(b)$. Although this integral algebra
is independent of $\mathbf i$, identifying the resulting seed
patterns for different expressions requires the comparison maps
to be compatible with braid moves. We do not prove this compatibility
here. In symmetric type, preservation of the integral form
does not identify the
bosonic global basis with the basis obtained from simple objects
in Theorem~\ref{thm:categorification}. That theorem realizes cluster
monomials by real simple objects; identifying them with elements
of the bosonic global basis requires a separate basis comparison.
\end{remark}

\begin{example}[A rank-one example]
Take \(\mathbf i=(i,i)\), set \(q_i=q^{d_i}\), and put
\(x=q_i^{1/2}f_{i,0}\), \(y=q_i^{1/2}f_{i,1}\).
Then
\[
xy-q_i^2yx=q_i(1-q_i^2),\qquad
g=xy-q_i,\qquad gx=q_i^{-2}xg.
\]
The word matrices and the corresponding variables are
\[
B=\begin{pmatrix}0\\-1\end{pmatrix},\qquad
\Lambda=\begin{pmatrix}0&2d_i\\-2d_i&0\end{pmatrix},\qquad
z_1=q_i^{-1/2}x,\quad z_2=q_i^{-2}g.
\]
The second variable is frozen. Mutation at the first position gives
\(z'_1=q_i^{-1/2}y\), since
\[
z_1z'_1=1+q_i z_2,\qquad z'_1z_1=1+q_i^{-1}z_2.
\]
These identities agree with the LS sign, the compatibility relation
\(B^{\mathsf t}\Lambda=(2d_i\ \ 0)\), and the integral
normalization used above. The resulting cluster algebra is generated
over \(\mathbb A\) by \(x,y\), since
\(z_2=q_i^{-2}(xy-q_i)\).
\end{example}

\subsection{Monoidal categorification}
\label{sec:monoidal-categorification}

Throughout the remainder of this section, assume that $C$ is symmetric
and choose $d_i=1$ for every $i$. The thickened Cartan matrix
$\widetilde C$ is then symmetric as well.

Let \(R_{\widetilde C}\) be a symmetric quiver Hecke algebra
in the standard geometric realization
over an algebraically closed field \(\mathbf{k}\) of
characteristic zero. We write \(\circ\) for convolution
and \(q\) for the grading-shift functor.
Let \(\mathscr C_{\widetilde w_{\mathbf i}}\) be the
standard full subcategory of finite-dimensional graded
\(R_{\widetilde C}\)-modules generated by the simple
modules corresponding to \(X_1,\ldots,X_{2r-1}\), under
convolution, grading shifts, extensions, and subquotients;
see \cite[Section~11.2]{kang2018monoidal}.

For a graded abelian category \(\mathscr C\), set
$$
K_{0,\A}(\mathscr C)
:=
\A\otimes_{\mathbb Z[q^{\pm1}]}K_0(\mathscr C),
$$
where \([qU]=q[U]\). Adjoining \(q^{1/2}\) extends the
coefficient ring; grading shifts in the category remain integral.
By integral quiver Hecke categorification, there is a root-graded
algebra isomorphism
\(\chi:K_{0,\A}
(\mathscr C_{\widetilde w_{\mathbf i}})
\xrightarrow{\sim}A_{\mathbf i,\A}\).

For \(1\le p<r\), let \(M_p\) be the determinantial
simple module corresponding to \(H_p\), up to a power of
\(q^{1/2}\). In particular,
\(\wt M_p=-\widetilde\beta_{2p}\).
As frozen determinantial modules, they have graded braiding
morphisms that are isomorphisms on
\(\mathscr C_{\widetilde w_{\mathbf i}}\)
\cite[Proposition~5.1 and Theorem~5.2]{KKOP-localization}.
The separators thus form a real commuting family of graded
braiders. By \cite[Theorem~2.12]{KKOP-localization}, we may form
the monoidal localization
\[
Q_{\mathbf i}:\mathscr C_{\widetilde w_{\mathbf i}}
\longrightarrow
\mathscr L_{\mathbf i}
:=\mathscr C_{\widetilde w_{\mathbf i}}
[M_p^{-1}\mid 1\le p<r].
\]
We also write \(M_p\) for its image and set
\(M^{\mathbf a}=M_1^{\circ a_1}\circ\cdots
\circ M_{r-1}^{\circ a_{r-1}}\), using this fixed order
and tensor inverses for negative exponents.

Every object is a finite direct sum of fractions
\(Q_{\mathbf i}(X)\circ M^{\mathbf a}\) with \(X\)
root-homogeneous. For \(X\in R_{\widetilde C}(\beta)\text{-gmod}\),
the extended weight and separator degree are
\[
\wt\bigl(Q_{\mathbf i}(X)\circ M^{\mathbf a}\bigr)
=-\beta-\sum_p a_p\widetilde\beta_{2p},
\qquad
\deg_\infty U=\deg_\infty(-\wt U).
\]
Let \(\mathscr C_{\mathbf i}^{0}\) be the full subcategory
of finite direct sums of fractions of separator degree zero.

\begin{lemma}[Localization at the separator subfamily]
\label{lem:separator-category-localization}
The category \(\mathscr L_{\mathbf i}\) is abelian of finite
length, its convolution is exact, and \(Q_{\mathbf i}\) is exact.
It sends simple modules to nonzero simple objects. Every simple
object of \(\mathscr L_{\mathbf i}\) is, up to an integer grading
shift, \(Q_{\mathbf i}(S)\circ M^{\mathbf a}\) for a simple
\(S\in\mathscr C_{\widetilde w_{\mathbf i}}\).
Moreover,
\begin{equation}\label{eq:separator-K0-localization}
K_0(\mathscr L_{\mathbf i})
\cong K_0(\mathscr C_{\widetilde w_{\mathbf i}})
\bigl[[M_p]^{-1}\mid 1\le p<r\bigr].
\end{equation}
The separator degree gives a direct-sum grading of
\(\mathscr L_{\mathbf i}\). Its zero component
\(\mathscr C_{\mathbf i}^{0}\) is an abelian monoidal subcategory,
and
\begin{equation}\label{eq:separator-K0-zero}
K_0(\mathscr C_{\mathbf i}^{0})
=K_0(\mathscr L_{\mathbf i})_0.
\end{equation}
\end{lemma}

\begin{proof}
The separators form a real commuting family, so the construction
of \cite[Theorem~2.12 and
Proposition~2.13]{KKOP-localization} gives abelianness and
exactness. Their braiding morphisms on
\(\mathscr C_{\widetilde w_{\mathbf i}}\) are isomorphisms. Thus
\cite[Proposition~4.8]{KKOP-localization} shows that simple
modules remain nonzero and simple, distinct simples remain
distinct, and the localized category has finite length.
The description of its simple objects follows from the fraction
presentation and exactness.
Tensoring with \(M^{\mathbf a}\) in the localized category is
an exact equivalence.

The classes \([M_p]\) are normal and pairwise \(q\)-commuting.
They are nonzero divisors, since scalar extension from
\(\mathbb Z[q^{\pm1}]\) to \(\A\) is faithfully free and
\(\chi\) identifies the extended Grothendieck ring with the domain
\(A_{\mathbf i,\A}\). Their monomials form an Ore denominator
set. The resulting localization maps to \(K_0(\mathscr L_{\mathbf i})\), and
fraction representatives give surjectivity. Both categories
have finite length, so their Grothendieck groups are free on
the isomorphism classes of simple objects. Since localization
preserves and separates simples, \(K_0(Q_{\mathbf i})\) is
injective. A fraction in the kernel
can be multiplied by its invertible denominator, reducing the
claim to this injectivity. This proves
\eqref{eq:separator-K0-localization}; it is the argument of
\cite[Corollary~5.4]{KKOP-localization} for the chosen subfamily.

The morphisms in the localization are obtained after tensoring
by common nonnegative powers of the separators. This preserves
the extended root weight, so morphisms
between distinct root weights vanish. Every object has a finite
root-weight decomposition, and kernels, cokernels, and exact
sequences decompose accordingly. Applying \(\deg_\infty\) gives
a direct-sum grading by \(\mathbb Z^{r-1}\). Since weights add
under convolution, the zero component
is abelian and monoidal. Taking Grothendieck groups gives
\eqref{eq:separator-K0-zero}.
\end{proof}

A simple object \(S\) is called \emph{real}
if \(S\circ S\) is simple.

\begin{theorem}\label{thm:categorification}
For every positive braid word \(\mathbf i\) in symmetric
type, there is an algebra isomorphism
$$
\mathcal A_{\A}(\mathbf s(\mathbf i))
\xrightarrow{\sim}
K_{0,\A}(\mathscr C_{\mathbf i}^{0})
$$
sending every quantum cluster monomial to the class of
a real simple object, up to a power of \(q^{1/2}\).
Thus \(\mathscr C_{\mathbf i}^{0}\) is a monoidal
categorification over \(\A\).
\end{theorem}

\begin{proof}
By Lemma~\ref{lem:separator-category-localization},
\(\chi\) extends to a root-graded algebra isomorphism
\(\chi_{\mathrm{loc}}:
K_{0,\A}(\mathscr L_{\mathbf i})
\xrightarrow{\sim}L_{\mathbf i,\A}\).
The coefficient extension from \(\mathbb Z[q^{\pm1}]\) to
\(\A\) is free, and the separator grading is a direct-sum
decomposition. Scalar extension therefore commutes with taking
separator degree zero, giving
$$
K_{0,\A}(\mathscr C_{\mathbf i}^{0})
\cong
\bigl(K_{0,\A}(\mathscr L_{\mathbf i})\bigr)_0
\cong
(L_{\mathbf i,\A})_0
=\D_{\mathbf i,\A}.
$$
The determinantial classes are normalized only by powers of
\(q^{1/2}\), which are units of \(\A\). The required isomorphism
therefore follows from
Theorem~\ref{thm:cluster-dehomogenization}.

Let \(t\) be an ambient seed, and let \(S_{t,k}\) be the
real simple module corresponding to
\(\widetilde Z_{t,2k-1}\), up to a power of \(q^{1/2}\).
By \cite[Theorem~11.2.3]{kang2018monoidal}, these modules
belong to a monoidal seed. In particular, for every
\(\mathbf a\in\mathbb N^r\), the convolution product
\(S_t^{\mathbf a}
=S_{t,1}^{\circ a_1}\circ\cdots\circ S_{t,r}^{\circ a_r}\)
is real simple.

Consider
\(T_{t,\mathbf a}
:=Q_{\mathbf i}(S_t^{\mathbf a})\circ M^{N_t\mathbf a}\).
It is simple because localization preserves simplicity
and tensoring with an invertible object is an equivalence.
Moreover, the separator objects commute up to grading
shifts with \(Q_{\mathbf i}(S_t^{\mathbf a})\), so
$$
T_{t,\mathbf a}\circ T_{t,\mathbf a}
\simeq
q^d Q_{\mathbf i}
 \bigl(S_t^{\mathbf a}\circ S_t^{\mathbf a}\bigr)
 \circ M^{2N_t\mathbf a}
$$
for some \(d\in\mathbb Z\).
The right-hand side is simple, so
\(T_{t,\mathbf a}\) is real.

By the defining property of \(N_t\),
$$
\deg_\infty T_{t,\mathbf a}
=
\sum_{k=1}^{r}a_k
\left(
\deg_\infty\widetilde Z_{t,2k-1}
+\sum_{p=1}^{r-1}n_{t,pk}\deg_\infty H_p
\right)
=0.
$$
Hence \(T_{t,\mathbf a}\in\mathscr C_{\mathbf i}^{0}\).
Its class satisfies
\(\chi_{\mathrm{loc}}([T_{t,\mathbf a}])
=q^{m/2}\mathsf M_t(\mathbf a)\)
for some \(m\in\mathbb Z\), by the normalization of quantum
monomials.
Proposition~\ref{pro:normalized-mutation} identifies
\(\mathsf M_t(\mathbf a)\) with the corresponding
cluster monomial of \(\mathcal A_{\A}(\mathbf s(\mathbf i))\).

Finally, let \(t\xrightarrow{k}t'\), write
\(\mathbf b=(b_{sk})_s\) for the \(k\)-th column of \(B_t\),
and put \(\mathbf c=N_t\mathbf b\). Set
\[
T_{t,s}:=Q_{\mathbf i}(S_{t,s})\circ M^{\mathbf n_{t,s}},
\qquad
E_\pm:=T_{t,1}^{\circ[\pm b_{1k}]_+}\circ\cdots
\circ T_{t,r}^{\circ[\pm b_{rk}]_+}.
\]
The ambient exchange sequence has middle term
\(S_{t,k}\circ S_{t',k}\) and end terms, up to integer
shifts, the ordered products with exponent vectors
\(([\pm\mathbf b]_+,[\pm\mathbf c]_+)\).
Apply the exact functor \(Q_{\mathbf i}\) and tensor by
\(M^{\mathbf n_{t,k}+\mathbf n_{t',k}}\). Equation~
\eqref{eq:normalization-mutation} gives
\[
[\pm\mathbf c]_++\mathbf n_{t,k}+\mathbf n_{t',k}
=N_t[\pm\mathbf b]_+.
\]
Using the braiding isomorphisms, we obtain
\begin{equation}\label{eq:normalized-categorical-exchange}
0\longrightarrow q^{d_1}E_{\varepsilon}
\longrightarrow T_{t,k}\circ T_{t',k}
\longrightarrow q^{d_2}E_{-\varepsilon}
\longrightarrow0
\end{equation}
for \(\varepsilon\in\{+,-\}\) and integers \(d_1,d_2\),
where \(-\varepsilon\) denotes the other sign. The order
of the end terms comes from the ambient exchange sequence;
reversing the convolution order reverses these terms.
The shifts combine those in the ambient sequence with those
in the braiding isomorphisms. All three terms have separator degree
zero. Taking classes and applying the toric normalizations of
Proposition~\ref{pro:normalized-mutation} gives the
quantum exchange relation. Hence the normalized families realize
mutation in \(\mathscr C_{\mathbf i}^{0}\).
Here the powers of \(q^{1/2}\) are coefficients in
\(K_{0,\A}\); the categorical grading shifts are integral.
\end{proof}

\subsection{An explicit thickening example}

\begin{example}[Thickening a nonreduced word in type $A_2$]
\label{ex:thickening-A2}
Let $C=\left(\begin{smallmatrix}2&-1\\-1&2\end{smallmatrix}\right)$
and let $\mathbf i=(1,2,2)$ represent the positive braid
$\sigma_1\sigma_2^2$. This word is not reduced in the Weyl group.
The thickened word is
$\mathbf j=(1,\infty_1,2,\infty_2,2)$.
In the ordered simple-root basis
$(\alpha_1,\alpha_2,\alpha_{\infty_1},\alpha_{\infty_2})$,
its root vectors have degrees given by the columns of
\[
 (\widetilde\beta_1\ \cdots\ \widetilde\beta_5)
 =\begin{pmatrix}
 1&2&5&16&27\\
 0&0&1&2&3\\
 0&1&2&6&10\\
 0&0&0&1&2
 \end{pmatrix}.
\]
Thus $H_1=q^{-1/2}X_2$ and $H_2=q^{-1/2}X_4$ have
separator degrees $(1,0)$ and $(6,1)$, respectively.  The
right-normalized PBW generators and their root degrees are
\[
\begin{aligned}
Y_1&=X_1,&\deg Y_1&=\alpha_1,\\
Y_2&=X_3H_1^{-2},&\deg Y_2&=\alpha_1+\alpha_2,\\
Y_3&=X_5H_1^2H_2^{-2},&\deg Y_3&=-\alpha_1-\alpha_2.
\end{aligned}
\]
The PBW and seed normalization matrices are
\[
 M=\begin{pmatrix}0&-2&2\\0&0&-2\end{pmatrix},
 \qquad
 N=MV=\begin{pmatrix}0&-2&0\\0&0&-2\end{pmatrix}.
\]
Write $\widetilde Z_s$ for the initial ambient seed variables,
so $\widetilde Z_2=H_1$, $\widetilde Z_4=H_2$,
$\widetilde Z_1=q^{-1/2}X_1$, and
$\widetilde Z_3=q^{-1/2}X_3$.
With products taken in the displayed order, the normalized
variables are
\[
 Z_1=\widetilde Z_1,\qquad
 Z_2=q^2\widetilde Z_3H_1^{-2},\qquad
 Z_3=\widetilde Z_5H_2^{-2}.
\]
Here only position $2$ is mutable, and
\[
 B_{\mathbf i}=\begin{pmatrix}0\\0\\-1\end{pmatrix},
 \qquad
 \Lambda_{\mathbf i}=
 \begin{pmatrix}0&-1&0\\1&0&2\\0&-2&0\end{pmatrix},
 \qquad
 B_{\mathbf i}^{\mathsf t}\Lambda_{\mathbf i}=(0\ \ 2\ \ 0).
\]
Let $\widetilde Z'_3$ be the ambient mutation at position $3$.
Its exchange column, in the order $(1,3,5,2,4)$, is $(0,0,-1,0,2)^{\mathsf t}$, and
\[
 \widetilde Z_3\widetilde Z'_3
 =q^{-2}H_2^2+q^{-1}\widetilde Z_5.
\]
The mutated normalization exponent is $(2,-2)^{\mathsf t}$.
The toric-frame convention gives
$Z'_2=q^{-20}\widetilde Z'_3H_1^2H_2^{-2}$;
using $H_1^{-2}\widetilde Z'_3
=q^{20}\widetilde Z'_3H_1^{-2}$, we obtain
\[
 Z_2Z'_2=1+qZ_3,\qquad Z'_2Z_2=1+q^{-1}Z_3.
\]
All scalar factors are units of $\A$, and the remaining
frozen variables $Z_1,Z_3$ are not inverted.

In the categorical realization, let $S_s$ and $S'_3$ represent
$\widetilde Z_s$ and $\widetilde Z'_3$, up to powers of $q^{1/2}$,
and let $M_1,M_2$ represent the separators.  The initial normalized
objects and their mutation are, up to integer grading shifts,
\[
\begin{aligned}
T_1&=Q_{\mathbf i}(S_1),&
T_2&=Q_{\mathbf i}(S_3)\circ M_1^{\circ(-2)},\\
T_3&=Q_{\mathbf i}(S_5)\circ M_2^{\circ(-2)},&
T'_2&=Q_{\mathbf i}(S'_3)\circ M_1^{\circ2}
                    \circ M_2^{\circ(-2)}.
\end{aligned}
\]
Their separator degrees are zero: those of
$S_3,S_5,S'_3$ are $(2,0),(12,2),(10,2)$, respectively.
The normalized exchange therefore takes place in
$\mathscr C_{\mathbf i}^{0}$.
\end{example}

\subsection{Embedding Theorem for Bao-Ye-M\'enard seed}
Let $v\leq w$, and let $\mathbf w$ be a reduced expression
of $w$. In symmetric type,
\cite{bi2026monoidal,kashiwara2026monoidal} give an embedding
\[
\Phi_{\mathbf w,v}:
\mathcal A_q(\mathbf s(v,\mathbf w))
\hookrightarrow K_{0,\mathbb K}(\mathscr C_{w,v})
\]
sending quantum cluster monomials to classes of simple
modules up to powers of $q^{1/2}$.
Since the integral cluster algebra is generated by its
cluster variables, this map restricts to an embedding
\[
\Phi_{\mathbf w,v,\mathbb A}:
\mathcal A_{\mathbb A}(\mathbf s(v,\mathbf w))
\hookrightarrow K_{0,\mathbb A}(\mathscr C_{w,v}).
\]

Now let $\mathbf i$ be a positive braid word and let
$v\leq\delta(\mathbf i)$. Put
$\mathbf j=\operatorname{th}(\mathbf i)$ and
$\widetilde w=\widetilde w_{\mathbf i}$.
All categories below are defined for the enlarged Cartan
matrix $\widetilde C$. We use their standard coordinate-ring
realizations and set
\[
R_{\widetilde w,v}
:=K_{0,\mathbb A}(\mathscr C_{\widetilde w,v})
\subseteq A_{\mathbf i,\mathbb A}.
\]
Let $M_p$ be the determinantial module corresponding,
up to a power of $q^{1/2}$, to the separator $H_p$.
As verified below, $M_p$ belongs to
$\mathscr C_{\widetilde w,v}$. Define
\[
\mathscr L_{\widetilde w,v}
:=\mathscr C_{\widetilde w,v}
[M_p^{-1}\mid 1\leq p<r],
\qquad
\mathscr C_{\mathbf i,v}^{0}
:=(\mathscr L_{\widetilde w,v})_0,
\]
where the subscript $0$ denotes auxiliary degree zero.
Write $S$ for the multiplicative set generated by the
separators and powers of $q^{1/2}$. By the proof of Lemma~\ref{lem:separator-category-localization}, the category $\mathscr C_{\mathbf i,v}^0$ is a abelian monoidal subcategory.

\begin{theorem}\label{thm:menard-seed-embedding}
Normalization induces an algebra embedding
\[
\Phi_{\mathbf i,v,\mathbb A}:
\mathcal A_{\mathbb A}(\mathbf s(v,\mathbf i))
\hookrightarrow
K_{0,\mathbb A}(\mathscr C_{\mathbf i,v}^{0})
=
\bigl(R_{\widetilde w,v}[S^{-1}]\bigr)_0.
\]
Moreover, $\Phi_{\mathbf i,v,\mathbb A}$ sends every
quantum cluster monomial to the class of a simple object
of $\mathscr C_{\mathbf i,v}^{0}$, up to a power of
$q^{1/2}$.
Extension of scalars gives an embedding
\(
\Phi_{\mathbf i,v}:
\mathcal A_q(\mathbf s(v,\mathbf i))
\hookrightarrow
K_{0,\mathbb K}(\mathscr C_{\mathbf i,v}^{0}).
\)
\end{theorem}

\begin{proof}
We first compare the seeds.
Since $v$ belongs to the original parabolic subgroup,
its leftmost reduced subexpression in $\mathbf j$ occurs
at $2p_1-1,\ldots,2p_m-1$. Thickening leaves the counts
$n_i,a_k,b_k$ unchanged, so the prescribed mutation
directions agree under $a\leftrightarrow 2a-1$.
Moreover, deleting the frozen separator rows commutes
with matrix mutation. Hence the vertices frozen or
deleted in the two constructions also agree.

After the final freezing, every deleted row is zero
on the retained mutable columns. These entries remain
zero under subsequent mutations, so the corresponding
exchange relations involve no deleted variables.
Proposition~\ref{pro:normalized-mutation} therefore gives
\[
Z_{t,a}
=
q^{c_{t,a}/2}\widetilde Z_{t,2a-1}
H^{\mathbf n_{t,a}},
\qquad
\deg_\infty Z_{t,a}=0,
\]
with the scalar prescribed by the toric normalization.
The normalization exponent vectors have an identity
block on the retained odd coordinates, so distinct
quantum torus monomials remain linearly independent.
Thus normalization induces an embedding
\[
\iota_{\mathbf i,v}:
\mathcal A_{\mathbb A}(\mathbf s(v,\mathbf i))
\hookrightarrow
\bigl(
\mathcal A_{\mathbb A}(\mathbf s(v,\mathbf j))[S^{-1}]
\bigr)_0.
\]

We next pass to the categorical realization.
Since $v\varpi_{\infty_p}=\varpi_{\infty_p}$, we have
\[
M_p
=
M(\widetilde w\varpi_{\infty_p},\varpi_{\infty_p})
=
M(\widetilde w\varpi_{\infty_p},v\varpi_{\infty_p})
\in\mathscr C_{\widetilde w,v}.
\]
The separator braiders restrict to this subcategory.
The argument of
Lemma~\ref{lem:separator-category-localization},
followed by taking auxiliary degree zero, gives
\[
K_{0,\mathbb A}(\mathscr C_{\mathbf i,v}^{0})
=
\bigl(R_{\widetilde w,v}[S^{-1}]\bigr)_0.
\]

The reduced-word embedding $\Phi_{\mathbf j,v}$ sends
each cluster variable to a power of $q^{1/2}$ times
a simple-module class, so it restricts to the integral
forms. It preserves auxiliary degree and extends
injectively after inverting the separators.
Composing this extension with $\iota_{\mathbf i,v}$
gives $\Phi_{\mathbf i,v,\mathbb A}$.

Let $X=\mathsf M_t(\mathbf a)$ be a quantum cluster
monomial, with $\mathbf a\in\mathbb N^{V_{\mathbf i}}$,
and let $\widetilde X$ be the corresponding ambient
cluster monomial in the retained odd variables.
Put
\(
\mathbf n
:=\sum_{a\in V_{\mathbf i}}a_a\mathbf n_{t,a}.
\)
By normalization,
\(
\iota_{\mathbf i,v}(X)
=q^{c/2}\widetilde XH^{\mathbf n}
\)
for some $c\in\mathbb Z$.
The reduced-word result
\cite{bi2026monoidal,kashiwara2026monoidal}
gives a simple module
$S\in\mathscr C_{\widetilde w,v}$ such that
$\Phi_{\mathbf j,v}(\widetilde X)=q^{d/2}[S]$.
Consequently,
\[
\Phi_{\mathbf i,v,\mathbb A}(X)
=
q^{e/2}
\bigl[Q_{\mathbf i,v}(S)\circ M^{\mathbf n}\bigr]
\]
for some $e\in\mathbb Z$.

By Lemma~\ref{lem:separator-category-localization},
$Q_{\mathbf i,v}(S)$ is nonzero and simple.
Since $M^{\mathbf n}$ is invertible, tensoring with it
is an equivalence, so
$Q_{\mathbf i,v}(S)\circ M^{\mathbf n}$ is also simple.
Moreover,
\[
\deg_\infty
\bigl(Q_{\mathbf i,v}(S)\circ M^{\mathbf n}\bigr)
=
\deg_\infty\widetilde X
+\sum_p n_p\deg_\infty H_p
=0.
\]
Thus this simple object belongs to
$\mathscr C_{\mathbf i,v}^{0}$, proving the assertion
for every quantum cluster monomial.
Finally, flatness of $\mathbb K$ over $\mathbb A$
gives the asserted embedding after scalar extension.
\end{proof}

\appendix
\section{Quantum seeds and monoidal localization}\label{sec:appendix}

Throughout this appendix, \(q\) is an indeterminate and
\(\mathbb K=\mathbb Q(q^{1/2})\). Unless otherwise specified,
quantum cluster algebras are defined over \(\A=\mathbb Z[q^{\pm1/2}]\),
and their frozen variables are not inverted.

\subsection{Quantum cluster algebras}
\label{app:quantum-cluster-algebras}

We use the conventions of \cite{BZ-quantum} for quantum
seeds and mutation.
Let \(K\) be a finite ordered index set, with mutable subset
\(K_{\mathrm{ex}}\) and frozen subset
\(K_{\mathrm{fr}}=K\setminus K_{\mathrm{ex}}\).
Let \(B=(b_{ij})_{i\in K,\,j\in K_{\mathrm{ex}}}\) be an
integer matrix with skew-symmetrizable principal part, and let
\(\Lambda=(\lambda_{ij})_{i,j\in K}\) be an integral
skew-symmetric matrix. The pair \((B,\Lambda)\) is
\emph{compatible} if there are positive integers
\(\rho_j\), \(j\in K_{\mathrm{ex}}\), such that
$$
B^{\mathsf t}\Lambda
=
(\rho_j\delta_{ji})_{j\in K_{\mathrm{ex}},\,i\in K},
\qquad\text{equivalently}\qquad
\sum_{s\in K}b_{sj}\lambda_{si}=\rho_j\delta_{ji}.
$$
For the word seeds below, \(\rho_k=2d_{i_k}\), where the
\(d_i\) symmetrize the Cartan matrix. In symmetric type
we use \(d_i=1\), so the diagonal coefficients are \(2\).

The based quantum torus \(\mathcal T^\Lambda\) is the
\(\A\)-algebra with basis
\(\{X^{\mathbf a}\mid\mathbf a\in\mathbb Z^K\}\) and
multiplication
$$
X^{\mathbf a}X^{\mathbf b}
=
q^{\frac12\mathbf a^{\mathsf t}\Lambda\mathbf b}
X^{\mathbf a+\mathbf b}.
$$
Writing \(X_i=X^{e_i}\), where \((e_i)_{i\in K}\) is the
standard basis of \(\mathbb Z^K\), one has
\(X_iX_j=q^{\lambda_{ij}}X_jX_i\) and
$$
X^{\mathbf a}
=
q^{-\frac12\sum_{i<j}a_i a_j\lambda_{ij}}
\prod_{i\in K}^{\longrightarrow}X_i^{a_i}.
$$
The product follows the fixed order on \(K\).

A \emph{toric frame} in a \(\mathbb K\)-division algebra \(\mathcal F\)
is a map
\(\mathsf M:\mathbb Z^K\to\mathcal F^\times\) identifying
a based quantum torus with a subalgebra of \(\mathcal F\)
whose division ring of fractions is \(\mathcal F\).
Its commutation matrix is denoted by \(\Lambda_{\mathsf M}\).
A \emph{quantum seed} is a pair
\(\mathbf s=(\mathsf M,B)\) such that
\((B,\Lambda_{\mathsf M})\) is compatible.
We also write
\(\mathbf s=((\mathsf M(e_i))_{i\in K},B,\Lambda_{\mathsf M})\)
when the index sets are understood.

For \(k\in K_{\mathrm{ex}}\), let \(\mathbf b_k\) be the
\(k\)-th column of \(B\), and write
\([\mathbf v]_+=(\max(v_i,0))_{i\in K}\).
The matrix mutation \(B'=\mu_k(B)\) is defined by
$$
b'_{ij}
=
\begin{cases}
-b_{ij},&i=k\text{ or }j=k,\\
b_{ij}+[b_{ik}]_+[b_{kj}]_+
       -[-b_{ik}]_+[-b_{kj}]_+,
       &i\ne k,\ j\ne k.
\end{cases}
$$
Let \(\mathsf E_k\) be the identity matrix with its
\(k\)-th column replaced by \(-e_k+[\mathbf b_k]_+\).
The mutated seed \(\mu_k(\mathbf s)=(\mathsf M',B')\)
has commutation matrix
\(\Lambda_{\mathsf M'}=
\mathsf E_k^{\mathsf t}\Lambda_{\mathsf M}\mathsf E_k\),
and its variables are
$$
\mathsf M'(e_k)
=
\mathsf M(-e_k+[\mathbf b_k]_+)
+\mathsf M(-e_k+[-\mathbf b_k]_+),
\qquad
\mathsf M'(e_j)=\mathsf M(e_j)\quad(j\ne k).
$$
These variables determine \(\mathsf M'\) by the normalized
monomial formula. The mutated pair is again compatible with
the same diagonal coefficients \(\rho_j\),
and mutation in direction \(k\) is involutive.

Write \(\mathbf s_t=(\mathsf M_t,B_t)\sim\mathbf s\)
for a seed obtained from \(\mathbf s\) by a finite sequence
of mutations. The \emph{quantum cluster algebra} used
in this paper is
$$
\mathcal A_{\A}(\mathbf s)
=
\A\left\langle
\mathsf M_t(e_i)
\mid \mathbf s_t\sim\mathbf s,\ i\in K
\right\rangle
\subseteq\mathcal F.
$$
The generators include the frozen variables, but not their inverses.
A \emph{quantum cluster monomial} is an element
\(\mathsf M_t(\mathbf a)\) with
\(\mathbf a\in\mathbb Z_{\ge0}^K\).

For comparison, define
\(\mathcal T_{\A}(\mathbf s_t)
=\A\langle\mathsf M_t(e_i)^{\pm1}\mid i\in K\rangle\).
The frozen localization and the upper quantum cluster
algebra are
$$
\mathcal A_{\A}^{\mathrm{loc}}(\mathbf s)
=
\mathcal A_{\A}(\mathbf s)
[\mathsf M(e_f)^{-1}\mid f\in K_{\mathrm{fr}}],
\qquad
\mathcal U_{\A}^{\mathrm{loc}}(\mathbf s)
=
\bigcap_{\mathbf s_t\sim\mathbf s}\mathcal T_{\A}(\mathbf s_t).
$$
The quantum Laurent phenomenon gives
\(\mathcal A_{\A}(\mathbf s)
\subseteq\mathcal A_{\A}^{\mathrm{loc}}(\mathbf s)
\subseteq\mathcal U_{\A}^{\mathrm{loc}}(\mathbf s)\subseteq\mathcal F\).
All frozen variables are inverted in
\(\mathcal U_{\A}^{\mathrm{loc}}\). We do not assert that it equals
the cluster algebra before inverting frozen variables. Extension
of scalars defines
\(\mathcal A_q(\mathbf s)=\mathbb K\otimes_{\A}
\mathcal A_{\A}(\mathbf s)\), realized inside the same skew field.

\subsection{Quantum seeds associated with words}
\label{app:word-seeds}

Let \(C=(c_{ij})_{i,j\in I}\) be symmetrizable, fix
\(d_i\in\mathbb Z_{>0}\) with \(d_ic_{ij}=d_jc_{ji}\), and let
\(\mathbf i=(i_1,\ldots,i_r)\) be a positive braid word,
not necessarily reduced in \(W\).
Put \(w_{\le k}=s_{i_1}\cdots s_{i_k}\), with \(w_{\le0}=1\), and
\[
k^-=\max\{a<k\mid i_a=i_k\},\qquad
k^+=\min\{a>k\mid i_a=i_k\},
\]
using \(\max\emptyset=-\infty\) and
\(\min\emptyset=+\infty\).
The frozen and mutable sets are
\(\operatorname{fr}(\mathbf i)=\{k\mid k^+=+\infty\}\)
and \(\operatorname{ex}(\mathbf i)
=[1,r]\setminus\operatorname{fr}(\mathbf i)\).

The extended exchange matrix
\(B_{\mathbf i}=(b_{st})_{s\in[1,r],\,
t\in\operatorname{ex}(\mathbf i)}\) is
\[
b_{st}=
\begin{cases}
-1,&s=t^+,\\
1,&s^+=t,\\
c_{i_si_t},&s<t<s^+<t^+,\\
-c_{i_si_t},&t<s<t^+<s^+,\\
0,&\text{otherwise}.
\end{cases}
\]
The principal part is skew-symmetrizable by
\(\operatorname{diag}(d_{i_k})_{k\in\operatorname{ex}(\mathbf i)}\).
In particular, the Cartan entries in the formula retain their
displayed row and column order.

Choose fundamental weights in a realization of \(C\), and
define the skew-symmetric matrix
\[
(\Lambda_{\mathbf i})_{st}
=\bigl((1-w_{\le s})\varpi_{i_s},
       (1+w_{\le t})\varpi_{i_t}\bigr)
\quad(s<t),
\qquad
(\Lambda_{\mathbf i})_{ss}=0.
\]
The root--weight pairing is normalized by
\((\alpha_i,\alpha_j)=d_ic_{ij}\) and
\((\alpha_i,\varpi_j)=d_i\delta_{ij}\).
Since \((1-w_{\le s})\varpi_{i_s}\in Q\), this expression
is integral and remains meaningful when \(C\) is singular.

The compatibility relation is
\[
\sum_{s=1}^r
b_{st}(\Lambda_{\mathbf i})_{su}
=2d_{i_t}\delta_{tu},
\qquad
t\in\operatorname{ex}(\mathbf i),\quad u\in[1,r].
\]

These signs agree with the forward LS convention
\(P_sP_t-q^{-(\gamma_s,\gamma_t)}P_tP_s
\in\langle P_{s+1},\ldots,P_{t-1}\rangle\) for \(s<t\).
By Proposition~\ref{pro:normalized-mutation}, the compatibility
relation can be written as
\[
B_{\mathbf i}^{\mathsf t}\Lambda_{\mathbf i}
=2(d_{i_k}\delta_{ks})_{k\in\operatorname{ex}(\mathbf i),\,s\in[1,r]}.
\]
Let \(M_{\mathbf i}\) be the standard toric frame in the
division ring of the based quantum torus
\(\mathcal T^{\Lambda_{\mathbf i}}\).
The \emph{quantum word seed} is
\(\mathbf s(\mathbf i)=(M_{\mathbf i},B_{\mathbf i})\).
We write \(\mathcal A_{\mathbb A}(\mathbf s(\mathbf i))\)
for its quantum cluster algebra over
\(\mathbb A=\mathbb Z[q^{\pm1/2}]\), with frozen variables
not inverted, and set
\(\mathcal A_q(\mathbf s(\mathbf i))
=\mathbb K\otimes_{\mathbb A}
\mathcal A_{\mathbb A}(\mathbf s(\mathbf i))\).
Theorem~\ref{thm:cluster-dehomogenization} realizes these
algebras in \(\D_{\mathbf i,\mathbb A}\) and
\(\D_{\mathbf i}\), respectively, via
\(M_{\mathbf i}(e_k)\mapsto Z_k\).

\subsection{Bao--Ye--M\'enard seeds}\label{app:menard-seed}

In this subsection assume that $C$ is symmetric and $d_i=1$ for all $i$.

Let $\mathbf i=(i_1,\ldots,i_r)$ be a positive word.
Its \emph{Demazure product} is defined recursively by
\begin{equation}\label{eq:demazure-product}
\delta(\emptyset)=e,
\qquad
\delta(\mathbf i i)
=\max\{\delta(\mathbf i),\delta(\mathbf i)s_i\},
\end{equation}
where the maximum is taken in the Bruhat order on $W$.
Fix $v\leq\delta(\mathbf i)$ and put $m=\ell(v)$.
The lexicographically smallest sequence
$1\leq p_1<\cdots<p_m\leq r$ satisfying
$s_{i_{p_1}}\cdots s_{i_{p_m}}=v$ determines the
\emph{leftmost reduced subexpression} of $\mathbf i$
representing $v$.

For $i\in I$, let $n_i$ be the number of occurrences of
$i$ in $\mathbf i$. We identify the position of its
$a$-th occurrence with the vertex $(i,a)$.
For $1\leq k\leq m$, set
\[
a_k:=\#\{s\leq k\mid i_{p_s}=i_{p_k}\},
\qquad
b_k:=\#\{t\leq p_k\mid
t\notin\{p_1,\ldots,p_m\},\ i_t=i_{p_k}\}.
\]
Thus $a_k$ and $b_k$ count the selected and unselected
occurrences of $i_{p_k}$ up to position $p_k$.

\begin{definition}\label{def:svbeta}
For $1\leq k\leq m$, put $i=i_{p_k}$ and define
\begin{equation}\label{eq:mu-tilde-k}
\widetilde\mu_k
:=\begin{cases}
\mu_{(i,n_i-a_k)}\circ\cdots\circ\mu_{(i,b_k+1)},
&a_k+b_k<n_i,\\[1mm]
\operatorname{Id},&a_k+b_k=n_i.
\end{cases}
\end{equation}
Composition is read from right to left, so the occurrence
indices are mutated in increasing order. Write
\begin{equation}\label{eq:M-k}
\mu_{\leq k}
:=\widetilde\mu_k\circ\cdots\circ\widetilde\mu_1,
\qquad
\mu_{\leq0}:=\operatorname{Id}.
\end{equation}
We use the same notation for the induced operations on
seeds, quivers, and compatible pairs.

Starting from $\mathbf s(\mathbf i)$, apply
$\widetilde\mu_k$ at step $k$ and then freeze
\(r_k:=(i_{p_k},n_{i_{p_k}}-a_k+1).
\)
After step $m$, set
\[
J_{\mathbf i}:=\{r_k\mid 1\leq k\leq m\},
\qquad
V_{\mathbf i}:=[1,r]\setminus J_{\mathbf i},
\]
using the identification of positions with occurrence
labels above. Retain the frozen status of every vertex
in $V_{\mathbf i}$, and freeze any remaining vertex
adjacent to $J_{\mathbf i}$ in the quiver
$\mu_{\leq m}(Q_{\mathbf s(\mathbf i)})$.
Then delete $J_{\mathbf i}$.
Denote the resulting frozen and exchangeable sets by
$V_{\mathbf i,\mathrm{fr}}$ and
$V_{\mathbf i,\mathrm{ex}}:=V_{\mathbf i}\setminus
V_{\mathbf i,\mathrm{fr}}$.

Let $(B_{\mathbf i}^{(m)},\Lambda_{\mathbf i}^{(m)})$
be the compatible pair obtained by applying
$\mu_{\leq m}$ to $(B_{\mathbf i},\Lambda_{\mathbf i})$,
before restricting to the retained vertices and
exchangeable columns. The freezing rule gives
\(
\bigl(B_{\mathbf i}^{(m)}\bigr)_
{J_{\mathbf i}\times V_{\mathbf i,\mathrm{ex}}}=0.
\)
Consequently, the restrictions
\begin{equation}\label{eq:restricted-compatible-pair}
B_{v,\mathbf i}
:=\bigl(B_{\mathbf i}^{(m)}\bigr)_
{V_{\mathbf i}\times V_{\mathbf i,\mathrm{ex}}},
\qquad
\Lambda_{v,\mathbf i}
:=\bigl(\Lambda_{\mathbf i}^{(m)}\bigr)_
{V_{\mathbf i}\times V_{\mathbf i}}
\end{equation}
form a compatible pair.

Let $y_a$ be the quantum cluster variable at vertex $a$
after applying $\mu_{\leq m}$. The
\emph{Bao--Ye--M\'enard seed} is
\begin{equation}\label{eq:seed-delta-beta}
\mathbf s(v,\mathbf i)
:=\bigl((y_a)_{a\in V_{\mathbf i}},
B_{v,\mathbf i},\Lambda_{v,\mathbf i},
V_{\mathbf i,\mathrm{ex}}\bigr).
\end{equation}
\end{definition}

We denote its quantum cluster algebra by
$\mathcal A_q(\mathbf s(v,\mathbf i))$, with no frozen
variables inverted.

\subsection{Localization of monoidal categories}
\label{app:monoidal-localization}

Assume that $C$ is symmetric and $d_i=1$ for all $i$.
Let $\mathbf k$ be an algebraically closed field of
characteristic zero, and let $R_C$ be the symmetric
quiver Hecke algebra associated with $C$.
We write $R_C\text{-gmod}$ for the category of finite
direct sums of finite-dimensional graded
$R_C(\beta)$-modules, with $\beta\in Q_+$.
Convolution, denoted by $\circ$, is exact in each argument
and has unit $\mathbf 1$. For $U\in R_C(\beta)\text{-gmod}$,
set $\operatorname{wt}U=-\beta$ and write $qU$ for its
grading shift. Weights add under convolution and are
unchanged by grading shifts.

For a root-homogeneous module $M\in R_C(\alpha)\text{-gmod}$,
define its left and right supports by
\[
\begin{aligned}
\mathbf W(M)
&:=\{\beta\in Q_+\mid
\alpha-\beta\in Q_+,\ e(\beta,\alpha-\beta)M\ne0\},\\
\mathbf W^*(M)
&:=\{\beta\in Q_+\mid
\alpha-\beta\in Q_+,\ e(\alpha-\beta,\beta)M\ne0\}.
\end{aligned}
\]
For $v,w\in W$ with $v\leq w$, let $\mathscr C_w$ and
$\mathscr C_{*v}$ be the full subcategories determined by
\[
\begin{aligned}
M\in\mathscr C_w
&\quad\Longleftrightarrow\quad
\mathbf W(M)\subseteq
\operatorname{span}_{\mathbb R_{\geq0}}
\bigl(\Delta^+\cap w\Delta^-\bigr),\\
M\in\mathscr C_{*v}
&\quad\Longleftrightarrow\quad
\mathbf W^*(M)\subseteq
\operatorname{span}_{\mathbb R_{\geq0}}
\bigl(\Delta^+\cap v\Delta^+\bigr).
\end{aligned}
\]
For an arbitrary object, these conditions are imposed on
each root-homogeneous component. Set
$\mathscr C_{w,v}:=\mathscr C_w\cap\mathscr C_{*v}$.

\begin{proposition}[{\cite[Proposition~2.16]{kashiwara2018monoidal}}]
The categories $\mathscr C_w$, $\mathscr C_{*v}$, and
$\mathscr C_{w,v}$ are closed under subquotients, extensions,
convolution products, and grading shifts.
\end{proposition}

For any of these categories $\mathscr C$, its Grothendieck
ring is the $\mathbb Z[q^{\pm1}]$-algebra with
$[U\circ V]=[U][V]$, $[qU]=q[U]$, and
$[V]=[U]+[W]$ for every short exact sequence
$0\to U\to V\to W\to0$. We use the scalar extensions
\[
K_{0,\A}(\mathscr C)
:=\A\otimes_{\mathbb Z[q^{\pm1}]}K_0(\mathscr C),
\qquad
K_{0,\mathbb K}(\mathscr C)
:=\mathbb K\otimes_{\mathbb Z[q^{\pm1}]}K_0(\mathscr C).
\]
Here $q^{1/2}$ is a scalar; categorical grading shifts
remain integral. The standard categorification identifies
$K_{0,\A}(\mathscr C_w)$ with the integral quantum
unipotent coordinate ring $A_q(\mathfrak n(w))_{\A}$.

A simple object $S$ is \emph{real} if $S\circ S$ is
simple. Two simple objects \emph{strongly commute} if
their convolution is simple. Every convolution monomial
in a pairwise strongly commuting family of real simple
quiver Hecke modules is again real simple
\cite{kang2018monoidal}.

A \emph{graded braider structure} on a root-homogeneous
object $M\in\mathscr C$ consists of a homomorphism
$\chi:Q\to\mathbb Z$ and natural morphisms
\[
\mathcal R_M(U):M\circ U
\longrightarrow q^{\chi(\operatorname{wt}U)}U\circ M
\]
for root-homogeneous $U$, compatible with the unit and
convolution. A family of graded braiders is \emph{real
commuting} if its self-braidings and double braidings
are nonzero scalar multiples of the identity, with
total grading shift zero.

Put $\widetilde w=\widetilde w_{\mathbf i}$.
We apply the localization construction to the separator
modules
\[
M_p:=M(\widetilde w\varpi_{\infty_p},\varpi_{\infty_p}),
\qquad 1\leq p<r,
\]
whose classes correspond to $H_p$ up to powers of
$q^{1/2}$. They form a subfamily of the frozen
determinantial modules in $\mathscr C_{\widetilde w}$,
and hence a real commuting family of graded braiders
whose braiding morphisms are isomorphisms
\cite[Proposition~5.1 and Theorem~5.2]{KKOP-localization}.

Localizing at this subfamily gives
\[
Q_{\mathbf i}:\mathscr C_{\widetilde w}
\longrightarrow
\mathscr L_{\mathbf i}
:=
\mathscr C_{\widetilde w}
[M_p^{-1}\mid 1\leq p<r].
\]
This localization is universal among graded monoidal
functors making the separator modules invertible.
The functor $Q_{\mathbf i}$ is exact, and
$\mathscr L_{\mathbf i}$ is abelian with exact convolution
\cite[Theorem~2.12 and Proposition~2.13]{KKOP-localization}.

For $\mathbf a\in\mathbb Z^{r-1}$, write
\[
M^{\mathbf a}
:=M_1^{\circ a_1}\circ\cdots\circ M_{r-1}^{\circ a_{r-1}},
\]
using the same notation for the separator modules and
their images under $Q_{\mathbf i}$.
Negative powers denote monoidal inverses.
By Lemma~\ref{lem:separator-category-localization},
the simple objects of $\mathscr L_{\mathbf i}$ are
precisely the objects
\[
q^d Q_{\mathbf i}(S)\circ M^{\mathbf a},
\qquad
d\in\mathbb Z,\quad
\mathbf a\in\mathbb Z^{r-1},
\]
where $S$ is simple in $\mathscr C_{\widetilde w}$.

The root grading extends by
$\operatorname{wt}(M_p^{-1})=-\operatorname{wt}(M_p)$
and induces the auxiliary grading on
$\mathscr L_{\mathbf i}$. Its degree-zero subcategory
is $\mathscr C_{\mathbf i}^{0}$.
On Grothendieck rings, the same lemma gives
\[
\begin{aligned}
K_{0,\A}(\mathscr L_{\mathbf i})
&\cong
K_{0,\A}(\mathscr C_{\widetilde w})
\bigl[[M_p]^{-1}\mid 1\leq p<r\bigr],\\
K_{0,\A}(\mathscr C_{\mathbf i}^{0})
&=
K_{0,\A}(\mathscr L_{\mathbf i})_0.
\end{aligned}
\]


\end{document}